\documentclass[11pt,reqno]{amsart}
\usepackage{amsthm,amsmath,amssymb}
\usepackage{graphicx}
\usepackage{float}
\usepackage{mathabx}
\usepackage[colorlinks=true,
linkcolor=blue,
urlcolor=blue,
citecolor=blue]{hyperref}
\usepackage{mathtools}
\usepackage{relsize}
\usepackage{tikz}
\usepackage{enumerate}
\usepackage[toc,page]{appendix}
\usepackage{pdflscape}
\usepackage{multirow}
\usepackage{adjustbox,expl3,etoolbox} 
\usepackage[margin = 3cm]{geometry}
\usepackage{rotating}
\usepackage{tablefootnote}
\usepackage{longtable}
\usepackage{comment}
\usepackage{xcolor}
\usepackage{mathrsfs}
\usepackage[all]{xy}

\usepackage{enumitem}

\newcommand\DN[1] {{\color{blue} {#1}}}

\newtheorem{thm}{Theorem}[section]
\newtheorem{lem}[thm]{Lemma}
\theoremstyle{definition}

\newtheorem{prop}[thm]{Proposition}

\newtheorem{cor}[thm]{Corollary}
\newtheorem{exa}[thm]{Example}
\newtheorem{rmk}[thm]{Remark}
\newtheorem{prob}[thm]{Problem}

\DeclareMathOperator\Aut{Aut}

\DeclareMathOperator{\soc}{soc}
\DeclareMathOperator{\Mon}{Mon}
\DeclareMathOperator{\mon}{Mon}
\DeclareMathOperator{\supp}{supp}
\DeclareMathOperator\PSL{PSL}
\DeclareMathOperator\AGL{AGL}
\DeclareMathOperator\PGL{PGL}
\DeclareMathOperator\ASL{ASL}
\DeclareMathOperator\GL{GL}
\DeclareMathOperator\SL{SL}

\newcommand{\F}{\mathbb{F}}
\newcommand{\Z}{\mathbb{Z}}

\newcommand\PGaL{{\rm P}\Gamma{\rm L}}
\newcommand\oline[1] {{\overline{#1}}}

\title{Monodromy groups of polynomials of composition length $2$ }

\author{Angelot Behajaina}
\address{Univ. Lille, CNRS, UMR 8524, Laboratoire Paul Painlevé, F-59000 Lille, France}
\email{angelot.behajaina@univ-lille.fr}

\author{Joachim König}
\address{Department of Mathematics Education, Korea National University of Education, Cheongju, South Korea}
\email{jkoenig@knue.ac.kr}

\author{Danny Neftin}
\address{Department of Mathematics, Technion - Israel Institute of Technology, Haifa, Israel}
\email{dneftin@technion.ac.il}

\begin{document}
	\begin{abstract}
		We study the monodromy groups of compositions of two indecomposable polynomials. In particular, we show that such monodromy groups either fulfill a certain ``largeness" property, or belong to an explicit list of exceptions. Such largeness results are crucial for dealing with compositions of more than two polynomials, and consequently are expected to have a wide range of applications to problems concerning the arithmetic of polynomials and arithmetic dynamics. In particular, our main result is a key ingredient in the solution of a long-standing open problem due to Davenport, Lewis and Schinzel, achieved in the companion paper \cite{BKN25}.
	\end{abstract}
	\maketitle
	\section{Introduction}
	Let $k$ be a field of characteristic $0$. The \emph{(arithmetic) monodromy group} of $f \in k[X] \setminus k$ over $k$, denoted by ${\rm Mon}_k(f)$ (or simply $\Mon(f)$ when there is no risk of confusion), is the Galois group of $f(X)-t$ over $k(t)$, viewed as a permutation group acting on the generic fiber $f^{-1}(t) \subset \overline{k(t)}$. The \emph{geometric monodromy group} of $f$ is the normal subgroup ${\rm Mon}_{\overline{k}}(f)\trianglelefteq \Mon_k(f)$.
	
	The classification of monodromy groups of indecomposable polynomials over $\mathbb{C}$ and $\mathbb{Q}$, along with their possible ramification types, was carried out by M\"{u}ller \cite{Mul2}. The next step is to address the following:
	\begin{prob}\label{qst:complength}
		Classify the monodromy groups of polynomials of length $2$ over $\mathbb{C}$ and
		$\mathbb{Q}$.
	\end{prob}
	\noindent
	Recall that the \emph{(composition) length} of $f \in k[X] \setminus k$ is the number of factors in a decomposition of $f$ as a composition of indecomposable polynomials.\footnote{By Ritt's theory (see \cite{MZ09}), this number is independent of the chosen decomposition.}

	We say that $f,g \in k[X]$ are \emph{linearly equivalent} (over $\overline{k}$) and write $f\sim g$, if there exist linear polynomials $\mu,\nu \in \overline{k}[X]$ such that $g=\mu \circ f \circ \nu$.  Linearly equivalent polynomials have the same geometric monodromy group.
	Note that an upper bound for the  monodromy group of a composition $f=g\circ h$ is given by (the imprimitive wreath product) $\Mon(h)\wr \Mon(g) = \Mon(h)^{\deg(g)}\rtimes \Mon(g)$. In the case of equality,  $\ker(\Mon(g\circ h)\to \Mon(g)) = \Mon(h)^{\deg(g)}$ is as large as possible. It turns out that, for many applications (see, e.g., \cite{KNR24} for applications in arithmetic dynamics), a slight weakening of this property is sufficient: namely, for $g,h\in k[X]$ of degree $>1$ with $h$ indecomposable, we say that $g\circ h$ has a {\it large kernel} if either
	$$
	\soc(\Mon_k(h))^{\deg(g)}\le \ker(\Mon_k(g\circ h)\to \Mon_k(g)),\footnote{For a group $G$, the \emph{socle} $\soc(G)$ denotes the subgroup generated by all minimal normal subgroups of $G$. Note that for an indecomposable $h$, the socle $\soc(\Mon_k(h))$ is always either a simple group or equal to $V_4\triangleleft S_4$.}$$
	or 
	$$\soc(\Mon_k(h))\,\,\textrm{is cyclic and}\,\, \soc(\Mon_k(h))^{\deg(g)-1}\le \ker(\Mon_k(g\circ h)\to \Mon_k(g)).$$
	
	Recall finally that the (normalized) \emph{Chebyshev polynomial} $T_n \in \Z[X]$ of degree $n \geq 1$ is uniquely determined by the identity $T_n(X+1/X)=X^n+1/X^n$. Chebyshev polynomials and monomials $X^n$ are at the other end of ``large kernel", fulfilling $\Mon_{\overline{k}}(T_n)=D_n$ ($n\ge 3$) and $\Mon_{\overline{k}}(X^n)=C_n$, and thus $\ker(\Mon_{\overline{k}}(T_n\circ T_m)\to \Mon_{\overline{k}}(T_n)) = \ker(\Mon_{\overline{k}}(X^n\circ X^m)\to \Mon_{\overline{k}}(X^n)) = C_m$. A more diverse source for violation of the large kernel property (cf.\ Proposition \ref{prop:rittimpldiag}) is pairs $f,g$ admitting a {\it Ritt move}; see Section \ref{sec:detailed_thms} for a definition. 
	
	Our main theorem states in a precise way that, for length-$2$ polynomials, the above examples are ``almost" the only ones violating the large kernel property:
	\begin{thm}\label{thm:mainres}
		Let $g,h\in k[X]$ be indecomposable polynomials of degree $>1$. 
		Then $g\circ h$ has a large kernel unless one of the following cases holds:
		\begin{enumerate}[label=\textbf{\arabic*.},ref=\arabic*]
			\item Over $\overline{k}$, one has $g\circ h\sim T_{p^2}$ or $g\circ h \sim X^{p^2}$ for some prime $p$;
			\label{mainthm_xntn}
			\item $g\circ h$ has a Ritt move;
			\label{mainthm_ritt}
			\item $\Mon(g\circ h)$ is one of the groups in Table \ref{table:sporcases}. In particular, one of the following holds:
			\label{mainthm_exc}
			\begin{enumerate}[label=\textbf{\alph*.}, ref=\alph*]
				\item $h\sim X^2$ and $\Mon_{\overline{k}}(g)\in \{S_4, { \PGL}_2(5), { \PSL}_3(2), { \PGL}_2(7),\\ \PGaL_2(9), { M}_{11}, { \PSL}_3(3), { \PSL}_4(2), \PGaL_3(4), { M}_{23}, {\rm \PSL}_5(2)\}$;
				\label{mainthm_exc_x2}
				\item  $h\sim X^3$ and $\Mon_{\overline{k}}(g)\in \{A_5, { \PGL}_2(7), { \PSL}_2(11), { \PSL}_3(3)\}$;
				\label{mainthm_exc_x3}
				\item $h\sim T_3$ and $\Mon_{\overline{k}}(g)=\PSL_3(3)$.
				\label{mainthm_exc_t3}
			\end{enumerate}
		\end{enumerate}
		In particular, if $\Mon(g),\Mon(h)$ are both solvable, then $g\circ h$ either has large kernel or is linearly equivalent to a monomial or a Chebyshev polynomial, or one of only two other cases holds: 
		\begin{itemize}
			\item $\Mon_{\overline{k}}(g \circ h) = C_3\times S_4\le S_{12}$, with a Ritt move;
			\item $\Mon_{\overline{k}}(g \circ h)=\GL_{2}(3)\le S_8$, without a Ritt move.
		\end{itemize}
	\end{thm}
	Further lower bounds on the kernel,  strengthening the bounds provided by Theorem \ref{thm:mainres} on its socle, are given in Section \ref{sec:detailed_thms}. These refinements extend the applicability of the theorem. The techniques used to prove this theorem are diverse and include permutation group theory, representation theory and linear algebra, as well as some topological ideas related to braid group action and configuration spaces of Galois covers. Some parts are also based on computer calculations in Magma \cite{Magma}, mainly using the database of transitive groups. Code for the nontrivial Magma verifications is contained in an ancillary file available at \url{https://arxiv.org/src/2603.27609v1/anc/Comp2_MagmaChecks.txt}.
	
	%These are particularly relevant since they allow for applications going} 
%Applications of 
The implications of Theorem \ref{thm:mainres} and a fortiori of its refinements go 
far beyond  considerations of compositions of only {\it two} indecomposables. We demonstrate this with a ``large kernel" conclusion for the (geometric) monodromy groups of arbitrary iterates $f^n$ of polynomials $f$ with dihedral monodromy group $D_p$. Apart from showing that these groups have maximal possible $p$-Sylow group (away from obvious exceptional polynomials), we also obtain strong lower bounds on the $2$-Sylow groups, thereby yielding a lower bound on the Hausdorff dimension in the dynamical limit $n\to \infty$.
Recall that $\Mon(f^n)$ is a subgroup of the $n$-fold iterated wreath product $[G]^n$ of $G=\Mon(f)$, and set $[G]^\infty=\varprojlim_n [G]^n$. For profinite groups $G\le H\le [S_d]^\infty$, the {\it relative Hausdorff dimension}\footnote{Since the ratio of the logarithms is independent of the base, the definition is unambiguous; we take base $p$ in the calculations below.} of $G$ in $H$ is  
$$\mathcal{H}_H(G) :=\liminf_{n\to\infty} \dfrac{\log |\pi_n(G)|}{\log |\pi_n(H)|},$$
where $\pi_n(G)$ is the image of $G$ under projection to $[S_d]^n$. 
\begin{thm}
	\label{cor:dynnom}
	Let $f\in k[X]$ be a polynomial of odd prime degree $p$ such that $f$ is linearly equivalent over $\overline{k}$, but not conjugate\footnote{Here $f,g$ are called conjugate over $\overline{k}$ if $g=\mu\circ f\circ \mu^{-1}$ for a linear $\mu\in \overline{k}[X]$.} to $\pm T_p$ over $\overline{k}$.
	Then $G:=\Mon_{\overline{k}}(f^\infty):=\varprojlim_n \Mon_{\oline k}(f^n)$ contains $[C_p]^\infty$, the infinite iterated wreath product of groups $C_p$. Moreover, the relative Hausdorff dimension of $G$ in $H:=[D_p]^\infty$ is at least
	$$\mathcal{H}_H(G) \ge
	1-\frac{\log_p(2)}{p(1+\log_p(2))}.$$
\end{thm}
The proof is given in Section \ref{sec:dyn_appl}. Further applications are expected in the field of arithmetic dynamics, and notably the study of {\it dynamical Galois groups}, i.e., the groups $G_{f,\alpha}:=\varprojlim_n \textrm{Gal}(f^n(X)-\alpha/k)$,
for $f\in k[X]$ and $\alpha\in k$. 
For $\alpha\in k$ which is not a critical value of iterates of $f$, these groups can be viewed as subgroups of $\Mon_k(f^\infty)=\varprojlim_n\Mon_k(f^n)$ and tend to be of small index in it. For $p=3$, this relation was  in particular studied in \cite{BT19, Ben} and \cite[Thm.\ 4]{women}, but much more remains to be explored.  

In \cite{BKN25}, the authors also apply (the restriction to solvable groups of) Theorem \ref{thm:mainres} to solve a long-standing open problem due to Davenport, Lewis and Schinzel \cite{DLS61} on the reducibility behavior of $f(X)-g(Y)$, with {\it arbitrary} polynomials $f,g$.

{\bf Acknowledgements}: The first and third authors were supported by the Israel Science Foundation, grant no.~353/21. The first author is also grateful for the support of a Technion fellowship, and an Open University of
Israel post-doctoral fellowship. He also acknowledges the support of the CDP C2EMPI,
as well as the French State under the France-2030 programme, the University of Lille, the Initiative of Excellence of the University of Lille, the European Metropolis of Lille for their
funding and support of the R-CDP-24-004-C2EMPI project. The second-named author was supported by the National
Research Foundation of Korea (NRF Basic Research Grant RS-2023-00239917).

\section{Beyond Theorem \ref{thm:mainres}}
\label{sec:detailed_thms}
In this section, we provide detailed refinements of Theorem \ref{thm:mainres}. The theorem follows directly from the combination of Theorems \ref{thm:affinesho}, \ref{thm:affS4I}, \ref{thm:affS4II}, \ref{thm:S4S4}, \ref{thm:arbnons}, \ref{thm:nonsaffI} and \ref{thm:nonsS4I} below, each of which deals with particular classes of monodromy groups $\Mon_k(g)$ and $\Mon_k(h)$. The possible classes are described by the monodromy classification for indecomposable polynomials, see Theorem \ref{prop:primmon}. Note that, due to the definition of the ``large kernel" property, we may work at the level of geometric monodromy groups in order to prove Theorem \ref{thm:mainres}. We therefore may and will assume for the rest of this section that $k=\overline{k}$ is an algebraically closed field of characteristic $0$.

\subsection{Solvable case} This part is devoted to the refinements of the solvable case of Theorem \ref{thm:mainres}. Note that Theorems \ref{thm:affinesho}, \ref{thm:affS4I}, \ref{thm:affS4II} and \ref{thm:S4S4} are essentially statements about configurations of branch points in a composition of two polynomials, as will become evident in the proofs. We additionally give explicit polynomials corresponding to each special case, whenever these are easy to compute from the respective branch point configurations; in some cases, however, such a parameterization would be too inconvenient to produce.

For a group $H$ and integer $n \geq 1$, let
$
{\rm diag}(H^n)=\{(h,\dots,h) \mid h \in H \}.
$
We shall call a polynomial $f$ an $\AGL_1$-polynomial (resp.\ nonsolvable polynomial, $S_4$-polynomial) if its monodromy group is a transitive subgroup of $\AGL_1(p)$ for a prime $p$ (resp.\ is nonsolvable, is $S_4$).  The {\it ramification type $E_f(P)$ of  $f$ over $P$}  is the tuple  of  ramification indices $e_f(Q/P)$ in decreasing order, where $Q$ runs over preimages in $f^{-1}(P)$. The ramification type of $f$ is then the multiset of tuples $E_f(P)$, where $P$ runs over branch points of $f$, cf.\ Section \ref{sec:setup}. 

\begin{thm}[Composition of ${\AGL}_1$-polynomials]\label{thm:affinesho}
	Let $g$ and $h$ be ${ \AGL}_1$-polynomials of prime degrees $q$ and $p$, respectively.
	Let $\Gamma={\ker}(\mon_k(g \circ h) \rightarrow \mon_k(g))$.
	\begin{enumerate}[label=\textbf{\arabic*.}, ref=\arabic*]
		\item Suppose $h\sim X^p$. Then $\Gamma=C_p^q$, unless one of the following holds:
		\begin{enumerate}[label=\textbf{\alph*.}, ref=\alph*]
			\item $g \circ h\sim X^{pq}$ -- in which case $\Gamma={\rm diag}(C_p^q)$.
			\item $p=2$ and $g \circ h\sim T_{2q}$ -- in which case $\Gamma={\rm diag}(C_2^q)$.
		\end{enumerate}
		\item Suppose $h\sim T_p$ for $p \geq 3$. Then $\Gamma=D_p^q$ or $\Gamma=\{ (a_k)_{k=1}^q \in D_p^q \mid a_1 \cdots a_q \in C_p \}$, unless $g \circ h\sim T_{pq}$ -- in which case 
		$$
		\Gamma=\begin{cases}
			{\rm diag}(C_p^q) &\textrm{if }q\neq 2,\\
			{\rm diag}(D_p^2) &\textrm{if }q=2.
		\end{cases}
		$$
	\end{enumerate}
\end{thm}

For the next statements, recall that there are two types of polynomials with monodromy group $S_4$: the {\it generic type}, of ramification type $([2,1^2], [2,1^2], [2,1^2], [4])$, and the {\it special type}, of ramification type $([3,1], [2,1^2], [4])$; see Theorem \ref{prop:primmon}. Following \cite{MZ09}, we say that  $\alpha \in k$ is  a \emph{special point} of   $f$ if it is unramified but lies over a branch point.

\begin{thm}[Composition of an $S_4$-polynomial and an ${ \AGL}_1$-polynomial]\label{thm:affS4I}
	Let $g \in k[X]$ be a polynomial with monodromy group $S_4$, and let $h$ be an $\AGL_1$-polynomial of prime degree $p$. Let $\Gamma={\ker}({\rm Mon}_k (g \circ h) \rightarrow {\rm Mon}_k(g))$. Then the following hold:
	\begin{enumerate}[label=\textbf{\arabic*.}, ref=\arabic*]
		\item If $g\circ h\sim (X^3+1)^3 X^3 = X^3(X-1)\circ (X^3+1)$, then $\Mon_k(g\circ h) = S_4\times C_3$. 
		\label{agls4_1_exc1}
		\item If $g\circ h\sim X^3(X-1)\circ (X^2+b)$, where $b$ is a root of $X^2+\frac{1}{2}X+\frac{3}{16}$, then $\Mon_k(g\circ h) = { \GL}_2(3)$ (central extension of $S_4$ by $C_2$).
		\label{agls4_1_exc2}
		\item In all other cases, $\Gamma\cap C_p^4$ has dimension at least $3$.
		\label{agls4_1_nonexc}
		Moreover, one has $\Gamma\supseteq C_p^4$ unless:
		\begin{enumerate}[label=\textbf{\alph*.}, ref=\alph*]
			\item $g\circ h\sim (X^2+\frac{3}{4})^3(X^2-\frac{1}{4}) = X^3(X-1) \circ (X^2+\frac{3}{4})$.
			\label{agls4_1_nonexc1}
			\item $g$ is special, $h\sim T_p$,
			and the two finite branch points of $h$ are both the preimages under $g$ of the finite branch point of ramification type $[3,1]$. In other words, $g\circ h \sim (X-2)^3(X+2)\circ T_p$.
			%% Second case is unnecessary, since equal to first case evaluated at -X (due to T_p(-X)=-T_p(X))
			\label{agls4_1_nonexc2}
			\item $h\sim T_p$, and the two finite branch points
			of $h$ are both the special points of $g$ lying over the same finite branch point of ramification type $[2,1^2]$. In other words,  $g\circ h \sim (X-a)^2(X^2-4)  \circ T_p$ for some $a\notin\{-2,0,2\}$.
			\label{agls4_1_nonexc3}
		\end{enumerate}
	\end{enumerate}
\end{thm}
For any $n \geq 2$ and $m \geq 2$, we let the \emph{augmentation subgroup} of $C_m^n$ be
$$
{\rm Aug}(C_m^n)=\{(a_1,\dots,a_n) \in C_m^n \mid a_1 \cdots a_n=1\}.
$$ 

\begin{thm}[Composition of an ${\AGL}_1$-polynomial and an $S_4$-polynomial]\label{thm:affS4II}
	Let $g$ be an $\AGL_1$-polynomial of prime degree $p$, and let $h \in k[X]$ be a polynomial with monodromy group $S_4$. Let $\Gamma={\ker}({\rm Mon}_k (g \circ h) \rightarrow {\rm Mon}_k(g))$. Then the following hold:
	\begin{enumerate}[label=\textbf{\arabic*.}, ref=\arabic*]
		\item If $g\circ h\sim (X^3+1)^3X^3 = X^3\circ ((X^3+1)X)$, then $\Mon_k(g\circ h) = S_4\times C_3$.
		\label{agls4_2_exc}
		\item In all other cases, $\Gamma$ contains $V_4^p$. 
		\label{agls4_2_nonexc} Moreover, $\Gamma=S_4^p$
		unless one of the following holds:
		\begin{enumerate}[label=\textbf{\alph*.}, ref=\alph*]
			\item \label{agls4_2_nonexc1} $g\sim X^p$,
			$h$ is special, and the finite branch point of $h$ of ramification type $[2,1^2]$ is the finite ramification point of $g$. In other words, $g\circ h \sim (X^3(X-4)+27)^p$. In this case, $\Gamma \supseteq A_4^p$.
			\item \label{agls4_2_nonexc2} $g\sim T_p$, $h$ is special, and the branch point of $h$ of ramification type $[2,1^2]$ is one of the two special points of $g$.
			In other words, $g\circ h \sim T_p\circ (aX\pm 2) \circ (X^3(X-4)+27)$ for some $a\ne 0$. In this case, $\Gamma/(\Gamma\cap V_4^p)\geq \mathrm{Aug}(C_3^p)$.
			\item $g \sim   X^p$ ($p \geq 5$), $h$ is generic, and the three finite branch points of $h$ map to the same point under $g$. In this case, one has $\Gamma \supseteq A_4^p$.
			\label{agls4_2_nonexc3}
			\item $g \sim T_p$, $h$ is generic, and the three finite branch points of $h$ are all ramified points of $g$ lying over the same branch point of $g$. In this case, one has $\Gamma \supseteq A_4^p$.
			\label{agls4_2_nonexc4}
			\item $g\sim T_3$, $h$ is generic, and $g$ maps the three branch points of $h$ all to the same non-branch point of $g$. Here $\Gamma/(\Gamma\cap V_4^3)$ contains $\mathrm{Aug}(C_3^3)$.
			\label{agls4_2_nonexc5}
			\item $g = \mu\circ  X^2\circ \nu$ for linear $\mu,\nu$, $h$ is generic, and the three finite branch points $u,v,w$ of $\nu\circ h$
			satisfy $u=0$ and $v^2=w^2$. In other words, $g\circ h \sim (X^2(X^2 + (\omega + 3)X + \frac{9}{8}\omega + \frac{27}{8}))^2$, where $\omega=\pm \sqrt{3}$. In this case, one has $\Gamma\supseteq A_4^2$.
			\label{agls4_2_nonexc6}
		\end{enumerate}
	\end{enumerate}
\end{thm}

\begin{thm}[Composition of two $S_4$-polynomials]\label{thm:S4S4}
	Let $g,h \in k[X]$ be $S_4$-polynomials. Let $\Gamma={\ker}({\rm Mon}_k (g \circ h) \rightarrow {\rm Mon}_k(g))$.
	Then $\Gamma$ contains $V_4^4\rtimes C_3^2$. Moreover, 
	$\Gamma$ contains $A_4^4$
	unless:
	\begin{enumerate}[label=\textbf{\arabic*.}, ref=\arabic*]
		\item Both $g$ and $h$ are special, and the branch points of $h$ of ramification types $[3,1]$ and $[2,1^2]$ are special points lying over the branch points of $g$ of ramification types $[3,1]$ and $[2,1^2]$ respectively. In other words,  $g\circ h \sim X(X+4)^3 \circ \alpha X(X+4)^3$, where $\alpha$ is a root of $X^2 - \frac{10}{27}X + \frac{1}{27}$.
		\item $g$ is special, $h$ is generic, and the three branch points of $h$ of ramification type $[2,1^2]$ are precisely the two points over the branch point of $g$ of ramification type $[3,1]$, together with one more special point over the branch point of $g$ of ramification type $[2,1^2]$.
		\item Both $g$ and $h$ are generic, and the three branch points of $h$ of ramification type $[2,1^2]$ are special points lying over pairwise distinct branch points of $g$ of ramification type $[2,1^2]$. 
	\end{enumerate}
\end{thm}

\subsection{Nonsolvable case}
This part is devoted to the refinements of the nonsolvable case of Theorem \ref{thm:mainres}.
Note that if the monodromy group of an indecomposable polynomial is nonsolvable, then this group is known by \cite{Mul2} to be either an alternating or symmetric group, or contained in an explicitly known finite list; cf.\ Theorem \ref{prop:primmon}.

\begin{thm}[Composition of an arbitrary indecomposable and a nonsolvable]
	\label{thm:arbnons}
	Let $g\in k[X]$ be an indecomposable polynomial, and let $h\in k[X]$ be a nonsolvable polynomial. Let $\Gamma=\ker(\Mon_k(g\circ h)\to \Mon_k(g))$. Then the following hold:
	\begin{enumerate}[label=\textbf{\arabic*.}, ref=\arabic*]
		\item If there is a Ritt move for $g\circ h$, then necessarily $g\sim X^p$, and moreover $\Mon_k(g\circ h) \cong C_p\times \Mon_k(h)$. \label{arbnons_exc}
		\item In all other cases, $\Gamma \supseteq \soc(\Mon_k(h))^{\deg(g)}$.
		\label{arbnons_nonexc}
	\end{enumerate}
\end{thm}

\begin{thm}[Composition of a nonsolvable with an $\AGL_1$-polynomial]\label{thm:nonsaffI}
	Let $g \in k[X]$ be a nonsolvable polynomial, and let  $h$ be an $\AGL_1$-polynomial of prime degree $p$. Consider the block kernel $\Gamma={\ker}({\rm Mon}_k (g \circ h) \rightarrow {\rm Mon}_k(g))$. Then the following hold:
	\begin{enumerate}[label=\textbf{\arabic*.}, ref=\arabic*]
		\item If there is a Ritt move for $g\circ h$, then necessarily $h\sim X^p$, and moreover $\Mon_k(g\circ h) \cong C_p\times \Mon_k(g)$.
		\label{nonsagl1_ritt}
		\item In all other cases, $\Mon_k(g\circ h)$ is either one of the nonsolvable groups in Table \ref{table:sporcases} of Appendix \ref{app:A}, or $\Gamma$ contains a subgroup $C_p^{\deg(g)-1}$. 
		\label{nonsagl1_noritt}
	\end{enumerate}
\end{thm}

\begin{thm}[Composition of a nonsolvable with an $S_4$-polynomial]\label{thm:nonsS4I}
	Let $g \in k[X]$ be a nonsolvable polynomial, and let $h \in k[X]$ be an $S_4$-polynomial. Consider the block kernel $\Gamma={\ker}({\rm Mon}_k (g \circ h) \rightarrow {\rm Mon}_k(g))$.
	Then $\Gamma$ contains $V_4^{\deg(g)}$. Moreover, $\Gamma$ contains a subgroup $C_3^{\deg(g)-1}$, unless one of the following holds:
	\begin{itemize}
		\item $\Mon_k(g)=\PSL_3(3)\le S_{13}$, with the $3$-Sylow group of $\Gamma$ of order $3^6$ or $3^{10}$.
		\item $\Mon_k(g)=\PSL_2(11)\le S_{11}$, with the $3$-Sylow group of $\Gamma$ of order $3^5$.
	\end{itemize}
\end{thm}

\section{Preliminaries}
In this section, we establish preliminary results needed for the sequel. Throughout the paper, $k$ is a field of characteristic $0$, and all group actions are left actions. We write $G=A.B$ to denote that $G$ is a (not necessarily split) group extension of $A$ by $B$. An  extension $C_p^n.B$ is written in short as $p^n.B$ and as $p^n\rtimes B$ if it is split.
\subsection{Setup} \label{sec:setup} 
A \emph{function field} $F$ over $k$ is a finite extension of $k(t)$, where $t$ is transcendental over $k$. Denote by $g_F$ its \emph{genus}. Let $F_1/F$ be an extension of function fields over $k$. For a place $Q$ of $F_1$ lying above a place $P$ of $F$, write $e(Q \mid P)$ for the ramification index  of $Q$ over $P$, cf.~\cite[Definition 3.1.5]{Sti09}. Let $Q_1, \dots, Q_r$ be the places of $F_1$ lying above a place $P$ of $F$. The multiset
$$
E_{F_1/F}(P) := [\, e(Q_1 \mid P), \dots, e(Q_r \mid P) \,]
$$
is called the \emph{ramification type} of $P$ in $F_1$. The place \(P\) is called a \emph{branch point} of $F_1/F$ if $e(Q_i \mid P) > 1$ for some $i$.  
Letting \(S\) be the set of branch points, we recall that $S$ is finite.  
The multiset
$
\{\, E_{F_1/F}(P) \mid P \in S \,\}
$
is called the \emph{ramification type} of $F_1/F$. 

Assume now that $F=k(t)$ and that the Galois closure of $F_1/k(t)$ has Galois group $G$, viewed as a transitive group of degree $n:=[F_1:k(t)]$. Write $S=\{P_1,\dots,P_u\}$. For each $P \in S$, the inertia group at $P$ is generated by some $\pi_P \in G$ of cycle type $E_{F_1/F}(P)$. 
Each $\pi_{P_{i}}$ is called the \emph{inertia group generator} at $P_i$, and the tuple $(\pi_{P_1}, \dots, \pi_{P_u})$ is also called the \emph{branch cycle description} of $F_1/k(t)$. It is well-defined up to simultaneous conjugation in $S_n$ (corresponding to relabeling of the set $\{1,\dots, n\}$) and {\it braid group action} (see Section \ref{sec:topol}), and satisfies $\pi_{P_1} \cdots \pi_{P_u}=1$.
The \emph{Riemann--Hurwitz formula} gives:
\begin{equation}
	\label{rh_formula}
	2 (n-1+g_{F_1}) =  \sum_{Q\,\,\text{place of } F_1} \left( e(Q \mid Q \cap k(t)) - 1 \right) = \sum_{P\in S} \textrm{ind}(\pi_P),
\end{equation}
where the {\it index} $\textrm{ind}(\pi)$ of $\pi\in S_n$ is defined as $n$ minus the number of orbits of $\langle\pi\rangle$.

A polynomial $f\in k[X]$ is decomposable if $f=g\circ h$ for $g,h\in k[X]$ of degrees $>1$. It is indecomposable if $\deg(f)>1$ and it is not decomposable. 
We say that $g \circ h$, for indecomposable $g,h \in k[X]$, is a \emph{non-unique decomposition} if there exist indecomposable $u,v \in k[X]$ such that $g \circ h = u \circ v$, but there is no linear polynomial $\eta \in k[X]$ with $g=u \circ \eta,h= \eta^{-1} \circ v$. It follows from Ritt's theorems (\cite{Ritt}; see \cite{MZ09} for a modern treatment)  that a non-unique decomposition is equivalent to the existence of a Ritt move. Recall that there is a \emph{Ritt move} for $g \circ h $ if $\gcd(\deg(g),\deg(h))=1$ and there exist indecomposable polynomials $u,v \in k[X]$ such that $$\deg(u)=\deg(g),\,\,\deg(v)=\deg(h),\,\, g \circ h=v \circ u,$$ and moreover, after possibly interchanging $(g,h)$ and $(v,u)$, there exist linear polynomials $\ell_1,\ell_2,\ell_3,\ell_4 \in k[X]$ such that the quadruple  
$$
\bigl(\, \ell_1 \circ g \circ \ell_2,\; \ell_2^{-1} \circ h \circ \ell_3,\; \ell_1 \circ v \circ \ell_4,\; \ell_4^{-1} \circ u \circ \ell_3 \,\bigr)
$$
is of one of the following types:

\begin{equation}
	\label{eq:rittstep}
	\begin{split}(T_n,\; T_m,\; &  T_m,\; T_n), \\
		(X^n,\; X^s w(X^n),\; & X^s w(X)^n,\; X^n),
	\end{split}
\end{equation}
where $m,n > 0$ are coprime, $s \geq 0$ is coprime to $n$, and $w \in k[X] \setminus Xk[X]$.

The \emph{affine general linear group of degree $1$ over a field $F$}, denoted by ${ \AGL}_1(F)$, is the semidirect product $F \rtimes F^\times$, where $F^\times$ acts by multiplication on $F$; this group acts naturally on $F$, where $F$ acts on itself by translation. An indecomposable polynomial $f \in k[X]$ is called an \emph{${ \AGL}_1$-polynomial} if $p=\deg(f)$ is prime and ${\rm Mon}(f)$ is solvable; equivalently, ${\rm Mon}_{\overline{k}}(f)$ is either cyclic or dihedral, that is, a permutation subgroup of ${ \AGL}_1(p)$.

\subsubsection*{Wreath products and permutation modules} Given two permutation groups $U \leq S_m$ and $V \leq S_n$, the \emph{wreath product} $U \wr V$ is the semidirect product $U^n \rtimes V$, where $V$ permutes the $n$ copies of $U$ via:
\begin{equation}\label{eq:multwr}
	v(u_1,\ldots,u_n)v^{-1}=(u_{v^{-1}(1)},\ldots,u_{v^{-1}(n)})\text{ for }v\in V,\text{ and }u_1,\ldots,u_n\in U.
\end{equation}

Denote by $\pi$ the projection onto $S_n$. If $G\leq U\wr S_n$ has transitive image under $\pi$, note that the images of the projections of $\Gamma=\ker(\pi) \cap G$ to the $n$ copies of $U$ are isomorphic, see \cite[Remark 3.2]{KN24}. Since the images are isomorphic, throughout the paper we shall repeatedly pick a projection arbitrarily and consider its image.

The \emph{support} of an element $x=(x_1,\dots, x_n)\in U^n$ is $\supp(x):=\{i\in \{1,\dots, n\} \mid x_i\ne 1\}$. In the following, for  an integer $n\geq 1$, we denote by $({\bf e}_1,\dots,{\bf e}_n)$ the standard basis of $\mathbb F_p^n$.

The following  useful lemma will be used implicitly many times when dealing with inertia generators of monodromy groups of polynomials. 
\begin{lem}\label{lem:direct}
	Assume that $\sigma=((\sigma_1,\dots,\sigma_n),\tau) \in S_m \wr S_n$ such that $\tau$ has order $k$. Then ${\rm Ker}(\pi) \cap \langle \sigma \rangle=\langle \sigma^k \rangle$, and
	$$
	\sigma^k=(\sigma_1 \sigma_{\tau^{-1}(1)}\dots\sigma_{\tau^{-(k-1)}(1)},\dots,\sigma_n \sigma_{\tau^{-1}(n)}\dots\sigma_{\tau^{-(k-1)}(n)}).
	$$
\end{lem}
\begin{proof}
	The equality ${\rm Ker}(\pi) \cap \langle \sigma \rangle=\langle \sigma^k \rangle$
	follows since $\pi(\sigma^i)=\tau^i$ is trivial if and only if $k\divides i$.
	Letting $\oline \sigma=(\sigma_1,\ldots,\sigma_n)$, one also has using \eqref{eq:multwr}:
	$$\sigma^k=\left(\prod_{i=0}^{k-1} (\tau^i \oline \sigma \tau^{-i})\right)\tau^k=\prod_{i=0}^{k-1}(\sigma_{\tau^{-i}(1)},\ldots,\sigma_{\tau^{-i}(n)})=\left(\prod_{i=0}^{k-1}\sigma_{\tau^{-i}(1)},\ldots,\prod_{i=0}^{k-1}\sigma_{\tau^{-i}(n)}\right).$$
\end{proof}
In practice, the  lemma will be used repeatedly together with the following observation.
\begin{rmk}\label{rmk:readramel}
	Assume that $k$ is algebraically closed. Let $g,h \in k[X]$ be polynomials of degrees $m$ and $n$, respectively, and let $u \in k$. Write $g^{-1}(u)=\{ v_1,\dots,v_s\}$. For each $i$, let $h^{-1}(v_i)=\{w_{i,1},\dots,w_{i,s_i}\}$. Let $e_i$ be the ramification index of $v_i$ above $u$ with respect to $g$,
	and let $e_{i,j}$ be the ramification index of $w_{i,j}$ above $v_i$ with respect to $h$. Then a generator of the inertia group of $g \circ h$ at $u$ has cycle type 
	$ (e_{i}e_{i,j})_{1 \leq i \leq s, 1 \leq j \leq s_i}.$
	Indeed, this follows directly from the multiplicativity of ramification indices and their correspondence with the cycle type of inertia generators in Section \ref{sec:setup}.
\end{rmk}

Note that, when $V\le S_n$ is a transitive group, $G\le \AGL_1(p)\wr V$ and $\Gamma=\ker(\pi: G\to V)$ is an elementary-abelian $p$-group, $\Gamma$ becomes a submodule of the $\mathbb{F}_p[V]$-module $C_p^n$. The fact that $V$ permutes the components setwise implies that the associated representation on $C_p^n$
is monomial. It is thus induced from  the point stabilizer $V_1\le V$ by a degree-$1$ representation $\psi: V_1\to \mathbb{F}_p^\times$, which describes the action of the block stabilizer $\pi^{-1}(V_1)$ on a block, cf., e.g., \cite[Exercise 43.1]{CurtisReiner66}. Moreover, the action of $\pi^{-1}(V_1)$ on the stabilized block factors through $\psi$ upon composing with the quotient $\AGL_1(\mathbb F_p)\to\mathbb F_p^\times$. When $\psi$ is the trivial representation, $C_p^n$ becomes isomorphic\footnote{Namely, via mapping ${\bf e}_i$ to a suitable generator of the $i$-th component of $C_p^n$.} to  the $\mathbb{F}_p[V]$-permutation module $\mathbb{F}_p \cdot {\bf e}_1 \oplus\dots\oplus \mathbb{F}_p \cdot {\bf e}_n$. Often this is enforced by the structure of $G$, e.g., when $G\le C_p\wr V$, or when $G\le D_p\wr V$, the image of the block action is $D_p$, and $V_1$ does not have a quotient isomorphic to $C_2\cong D_p/C_p$. 

In particular, in the case where additionally $V$ is a cyclic group of order $n$ and $\langle \sigma\rangle \le G$ is a preimage of $V$ under $\pi$ with $\sigma^n\in C_p^n$, by the above, the $\mathbb{F}_p[V]$-(permutation) module $\mathbb{F}_p^n$ becomes the cyclic code $\F_p[x]/(x^{n}-1)$. By a cyclic code we shall henceforth refer to an ideal in $\mathbb F_p[x]/(x^n-1)$.

Due to this special importance of the permutation module, the action of  $G\le S_n$  on a direct product $U^n$ should always be understood as the action via permutation of components. In particular, when $U=\mathbb{F}_p$ and $\textbf{v}\in U^n$, the module $\langle G\cdot \textbf{v}\rangle$ shall denote the submodule of the $\mathbb{F}_p[G]$-permutation module generated by $\bf{v}$.

\subsubsection*{Fiber products} Given groups $G_1, G_2, H$, and morphisms $\phi_1: G_1 \to H,\phi_2: G_2 \to H$, the \emph{fiber product} of $G_1$ and $G_2$ over $H$ is the subgroup of $G_1 \times G_2$ given by
$$
G_1 \times_H G_2 := \{ (g_1, g_2) \in G_1 \times G_2 \mid \phi_1(g_1) = \phi_2(g_2) \}.
$$

\subsection{Monodromy groups of indecomposable polynomials}
We recall the main known results on indecomposable polynomials and their monodromy groups.
The following essentially summarizes the classification results of \cite{Mul2} (together with some more elementary considerations regarding the possibilities for indecomposables with solvable monodromy group, cf.\ \cite[Lemma 2.9]{Mul}).

\begin{thm}
	\label{prop:primmon}
	Let $k$ be an algebraically closed field of characteristic $0$, and let $f\in k[X]$ be an indecomposable polynomial of degree $>1$. Then one of the following holds:
	\begin{enumerate}[label=\textbf{\arabic*.}, ref=\arabic*]
		\item $\Mon_k(f)$ is solvable, and moreover one of the following holds.
		\label{primmon_solv}
		\begin{enumerate}[label=\textbf{\alph*.}, ref=\alph*]
			\item  $f \sim X^p$ for a prime $p$. Here, the ramification type of $f$ is $([p], [p])$. 
			\label{primmon_xp}
			\item 
			\label{primmon_tp}
			$f\sim T_p$ for a prime $p\ge 3$. Here, the ramification type of $f$ is $([2^{(p-1)/2},1], [2^{(p-1)/2},1], [p])$. 
			\item $\deg(f)=4$ and $\Mon_k(f)=S_4$. Here either
			\label{primmon_s4}
			\begin{enumerate}[label=\textbf{\roman*.}, ref=\roman*]
				\item $f\sim X^3(X-1)$ is special (of  ramification type $([3,1], [2,1^2], [4])$), or
				\label{primmon_s4_special}
				\item $f\sim X^2(X^2+aX+a)$, for $a\in k\setminus\{0,\frac{32}{9},4\}$,\footnote{The exclusion of the three values $a=0$, $a=\frac{32}{9}$ and $a=4$ is due to the fact that these (and only these, as a quick discriminant calculation shows) decrease the number of branch points, with exactly the value $\frac{32}{9}$ leading to an indecomposable, namely as in Case \ref{primmon_solv}\ref{primmon_s4})\ref{primmon_s4_special}.} is generic (with ramification type $([2,1^2], [2,1^2], [2,1^2], [4])$).
				\label{primmon_s4_generic}
			\end{enumerate}
		\end{enumerate}
		\item \label{primmon_nonsolv}
		$\Mon_k(f)$ is a nonsolvable almost simple group, and moreover one of the following holds.
		\begin{enumerate}[label=\textbf{\alph*.}, ref=\alph*]
			\item \label{primmon_ansn}
			$\Mon_k(f)=S_n$ ($n\ge 5$) or $A_n$ ($n\ge 5$ odd), with many possible ramification types.
			\item $\Mon_k(f)\in \{\PGL_2(5), \PSL_3(2), \PGL_2(7), \PGaL_2(8), \PGaL_2(9), \PSL_2(11), M_{11}$,\\  $\PSL_3(3), \PSL_4(2), \PGaL_3(4), M_{23}, \PSL_5(2)\}$, of degrees $6,7,8,9,10,11,11,13,15,21,23$ and $31$, respectively. Here, all possible ramification types for $f$ are explicitly given in \cite[\S 2.2, Thm.]{Mul2}.
			\label{primmon_exc}
		\end{enumerate}
	\end{enumerate}
\end{thm}

\subsection{Kernel estimates}\label{sec:linalg}
In this section we collect various auxiliary results which will prove useful when lower-bounding the kernel $G\cap \ker(\pi)$ for transitive groups $G\le U\wr V$ (where $\pi:U\wr V\to V$ is the projection). In particular these hold in the various cases where $U$ is one of the solvable primitive monodromy groups occurring in Theorem \ref{prop:primmon}.

\subsubsection{On groups $G\le U\wr V$ with $U$ cyclic}
\label{sec:uv_ucyclic}

The following results are useful to obtain ``large kernel" conclusions for subgroups of $C_p\wr C_q (=\mathbb{F}_p^q\rtimes C_q)$, $C_p\wr D_q$ or $C_p\wr S_4$, under the assumption that the kernel contains a certain ``short vector".

\begin{lem} %Applied in Proof of Thm. 2.1
	\label{lem:Szpee}
	Let $p,q$ be primes with $p\not\equiv 1 \bmod q$, and let ${\bf v} = {\bf e}_i + \mu {\bf e}_j\in \mathbb{F}_p^q$, with $i\ne j$ and with $\mu\in \mathbb{F}_p^\times$.
	Then $\langle C_q\cdot {\bf v}\rangle = \begin{cases}{\rm{Aug}}(\mathbb{F}_p^q), \text{ if } \mu=-1,\\ \mathbb{F}_p^q, \textit{ otherwise}.\end{cases}$
\end{lem}
\begin{proof} Due to transitivity of $C_q$ in the permutation action, we may assume $j=0$ without loss of generality.
	The cyclic code $\langle C_q\cdot {\bf v}\rangle$ thus has generator polynomial $\gcd(X^q-1, X^i+\mu)$. Any root of $X^i+\mu$ in $\overline{\mathbb{F}}_p$ has multiplicative order dividing $i(p-1)$, and since $i<q$ and $p\not\equiv 1\bmod q$, this order is coprime to $q$. Hence, $\gcd(X^q-1, X^i+\mu)$ cannot have any irreducible factor apart from $X-1$, which also cannot be a multiple factor since one of $X^q-1$ and $X^i+\mu$ is separable. Clearly $X-1 \divides X^i+\mu$ if and only if $\mu=-1$. The assertion now follows via noting that the polynomial $X-1$ generates the augmentation ideal. 
\end{proof}

The next lemma gives a similar conclusion for cyclic groups $C_q$ not necessarily of prime order.
\begin{lem}
	\label{lem:Szodd}%Applied in Proof of Prop.6.2, but might want to develop there directly
	Let $p$ be a prime, and $q\ge 2$ an odd integer.  If $p=2$ assume further $q$ is prime.
	For ${\bf v}={\bf e}_j + {\bf e}_{i+j}\in \mathbb{F}_p^q$, $i\ne 0$, one has:
	$$\langle C_q \cdot {\bf v} \rangle=\begin{cases}
		{\rm Aug}(C_2^{q}), &\text{ if } p=2,\\
		\mathbb{F}_p^q, &\text{ otherwise}.\end{cases}$$
\end{lem}
\begin{proof}
	The subgroup $\langle C_q {\cdot} {\bf v}\rangle$ corresponds to the cyclic $\mathbb{F}_p$-code of length $q$ with generator polynomial 
	$$\gcd(X^q-1, X^i+1) = 
	\begin{cases} X+1, &\text{ if }  p=2,\\
		1, &\text{ else}.
	\end{cases}$$
	In the former case, $\langle C_q{\cdot} {\bf v}\rangle={\rm Aug}(C_2^q)$. In the latter case,  $\langle C_q{\cdot} {\bf v}\rangle = \mathbb{F}_p^q$. 
\end{proof}

The next results are useful for analyzing groups $G\le U\wr V$ with $V$ dihedral, resp., $V=S_4$.
\begin{lem} %Applied in proof of "support-3" Corollary, and in proof of Thm.2.1. and Thm.2.3.
	\label{lem:ffkf_new}
	Let $q\ge 5$ be a prime, and ${\bf v}\in \mathbb{F}_2^q$ a vector such that $\supp({\bf v})$ is of size $4$ and invariant under some reflection of $D_q$. Then $\langle C_q\cdot {\bf v}\rangle = {\rm{Aug}}(\mathbb{F}_2^q)$.
\end{lem}
\begin{proof}
	Invariance under a reflection implies that $\supp({\bf v})$ is of the form $\{i_0, i_0+i, j,j+i\}$ for suitable $i_0,i,j\in \mathbb{F}_q$. We may assume $i_0=0$ and $0<i<j<q$ without loss of generality.
	The subgroup $\langle C_q\cdot {\bf v}\rangle$ thus corresponds to a cyclic code with a generator polynomial $\gcd(X^q-1, 1+X^i+X^j+X^{j+i}) = \gcd(X^q-1, (X^j-1)(X^i-1)) $ which is $X-1$ since $i,j$ are coprime\footnote{For odd $i$ coprime to $q$, any $q$-th root of unity in $\oline{\mathbb F}_2$ which is also an $i$-th root of unity is $1$. The claim then follows when noting that $2$-powers in $i$ only affect multiplicity while $x^q-1$ is square free.} to $q$. Hence, $\langle C_q\cdot {\bf v}\rangle = \textrm{Aug}(\mathbb{F}_2^q)$ as claimed.
\end{proof}

\begin{cor}
	\label{cor:dp_smallweight}
	Let $q\ge 5$ be a prime, and let ${\bf v} \in \F_2^q$ with $|\supp({\bf v})| = 3$. Then $\langle D_q \cdot {\bf v} \rangle = \F_2^q$.
\end{cor}
\begin{proof}
	Write $\supp({\bf v}) =\{i,j,k\}\subset \mathbb{F}_q$. Since $q>3$, we may assume (after reordering) that $j-i\ne k-j$. Without loss of generality, we may assume $j=0$. Let $\sigma\in D_q$ be the reflection of coordinates $d\mapsto -d$, $d\in \mathbb{F}_q$. 
	Then 
	$$
	{\bf w}:={\bf v} + \sigma\cdot {\bf v}={\bf e}_i+{\bf e}_{-i}+{\bf e}_{k}+{\bf e}_{-k}
	$$  has a $\sigma$-invariant support. Lemma \ref{lem:ffkf_new} implies that $\langle C_q \cdot {\bf w}\rangle$ contains ${\textrm{Aug}}(\mathbb{F}_2^q)$. Moreover, as $\langle D_q \cdot {\bf v}\rangle\supseteq \langle C_q \cdot {\bf w}\rangle$ and ${\bf v} \notin {\rm Aug}(\F_2^q)$, it follows that $\langle D_q \cdot {\bf v} \rangle=\F_2^q$.
\end{proof}

\begin{lem}\label{lem:havzero}
	Let ${\bf v} \in \F_p^4$ with $1 \leq |\supp({\bf v})| \leq 3$. Then $\langle S_4 \cdot {\bf v} \rangle$ contains ${\rm Aug}(\F_p^4)$.
\end{lem}
\begin{proof}
	Write $v=(v_1,\ldots,v_4)$. Without loss of generality, we may assume that $v_1 \neq 0$ and $v_2=0$. Then
	$$
	{\bf e}_1-{\bf e}_2=\frac{{\bf v}-(1,2)\cdot {\bf v}}{v_1} \in \langle S_4 \cdot {\bf v} \rangle, 
	$$
	which implies that ${\rm Aug}(\F_p^4) \leq \langle S_4 \cdot {\bf v} \rangle$.
\end{proof}
We shall also use the fact that monodromy groups of polynomials necessarily contain a cyclic transitive subgroup  to  obtain lower bounds for kernels.
\begin{lem}\label{lem:fullcyindiag}
	Let $\sigma$ be an $mn$-cycle in $C_n \wr S_m$ with $m,n \geq 2$. Then $\sigma^m\in {\rm diag}(C_n^m)\setminus\{1\}$.
\end{lem}
\begin{proof}
	Write $\sigma=((a_1,\dots,a_m),\tau)$, where $a_i \in C_n$ and $\tau\in S_m$ is an $m$-cycle. By  Lemma \ref{lem:direct},
	$$
	\sigma^m=\left(\prod_{i=1}^ma_i,\dots,\prod_{i=1}^ma_i\right) \in {\rm diag}(C_n^m).
	$$ Since $\sigma$ is an $mn$-cycle, $\sigma^m$ is nontrivial.
\end{proof}

\subsubsection{On groups $G\le U\wr V$ with $U$ non-cyclic}
\label{sec:uv_unoncyclic}
The following is useful to obtain ``large kernel" conclusions for the case of subgroups $G\le U\wr V$ with certain {\it noncyclic} groups $U$, and notably when $U$ is a noncyclic solvable monodromy group of an indecomposable polynomial. 
\begin{lem}\label{lem:normal}
	Let $d \geq 3$.
	\begin{enumerate}
		\item
		Let $q$ be an odd prime, $G\le \AGL_1(q)\wr S_d$ a transitive subgroup, and $K:=\ker(G\to S_d)$. Assume that there exist $s\geq 1$ elements $x_1,\dots, x_s\in K$  of order coprime to $q$ such that $|\bigcap_{j=1}^s\supp(x_j)| = 1$.
		Then $K$ contains $C_q^d$ as a minimal normal subgroup.
		\item Let $G\le S_4\wr S_d$ be a transitive subgroup, and $K:=\ker(G\to S_d)$.
		Assume  that there exist $s\geq 1$ elements $x_1,\dots, x_s\in K$  of order $3$ such that $|\bigcap_{j=1}^s\supp(x_j)| = 1$.
		Then $K$ contains $V_4^{d}$ as a minimal normal subgroup. 
	\end{enumerate}
\end{lem} 

Lemma \ref{lem:normal} is a special case of the following more general statement, which is of interest in its own right, even though not used in full generality in this paper.

\begin{lem}
	\label{lem:fullsocle}
	Let $H=W.U$,
	where $W$ is a faithful $\mathbb{F}_p[U]$-module.
	Let $G\le H\wr S_d$ have transitive image under $\pi:G\to S_d$. Let $\Gamma=\ker(\pi) = G\cap H^d$ and $\Delta  = \Gamma\cap W^d$. Furthermore, let $\rho:\Gamma\to \GL(W)^d$ be the  action of $\Gamma\leq H^d$ on $W^d$. Assume all of the following:
	\begin{itemize}
		\item[i)] The image of $\Gamma$
		under projection to (any) $H$-component is $W.U'$ for  $U'\le U$ such that  $W=W_1\oplus\dots \oplus W_r$ for irreducible $\mathbb{F}_p[U']$-modules $W_j$, $j=1,\dots, r$. \\ Let $\rho_j:\Gamma\to \GL(W_j)^d$ be the  action of $\Gamma\leq H^d$ on $W_j^d$.
		\item[ii)] \label{it:support}
		There exist $s\ge 1$ elements $x_1,\dots,x_s\in \Gamma$ such that the supports of $\rho(x_1),\ldots,\rho(x_s)$ admit an intersection $\bigcap_{i=1}^s\supp(\rho(x_i))$ of size $1$.  Let $J$ be the set of all $j$ for which the intersection $\bigcap_{m=1}^s\supp(\rho_j(x_m))$ is nonempty (and hence of size $1$).
	\end{itemize} 
	Then $\Delta\supseteq W'^d$, where $W':= \bigoplus_{j\in J} W_j$. Moreover, if $r=1$, then $W^d$ is a minimal normal subgroup of $G$.
\end{lem}
Recall from Section \ref{sec:setup} that the image of $\Gamma$ under each of the $H$-components is isomorphic, so we may pick  $W.U'$ to be any of these. 
\begin{proof}
	It suffices to show $\Delta\supseteq W_j^d$, $j\in J$. Henceforth fix $j\in J$. Let $x_1,\dots, x_s$ be as in Assumption ii), and let $i \in \{1,\dots, d\}$ denote the unique element of $\bigcap_{m=1}^s\supp(\rho(x_m))$.
	\\
	\underline{Observation 1}: Given any $\gamma \in \Gamma$ whose $i$-th component $\gamma_i$ lies in $W_j\setminus\{0\}$, and any $1\ne u\in U$ acting nontrivially on $W_j$, there exists a $\Gamma$-conjugate $\tilde{\gamma}\in \Gamma$ of $\gamma$ such that $\tilde{\gamma}_i$ lies in $W_j$ {\textit and} is not fixed by $u$ (i.e., $\tilde{\gamma}_i^u\ne \tilde{\gamma}_i)$. Indeed, irreducibility of $W_j$ implies that the module generated by $\gamma_i^g$, $g\in W.U'$  is all of $W_j$, and the nontrivial action of $u$ then implies that there exists $g\in W.U'$ such that $\gamma_i^g$ is not fixed by $u$. It thus suffices to take $\tilde{\gamma} = \gamma^h$, where $h \in \Gamma$ is an element with $i$-th component entry equal to $g$, which is possible by Assumption i).

	\underline{Claim 1}: By induction on $k=0,\dots, s$, there exists $z^{(k)}\in \Gamma$ whose $i$-th component $z^{(k)}_i$ lies in $W_j\setminus\{0\}$ and such that  $\supp(\rho(z^{(k)})) \subseteq (\bigcap_{m=1}^k \supp(\rho(x_m)))\setminus\{i\}$, where $\bigcap_{m=1}^0 \supp(\rho(x_m))=\{1,\dots,d\}$.
	
	The base case $k=0$ follows since the image of $\Gamma$ under projection to the $i$-th component contains $W_j$ due to Assumption i). Assume inductively that the claim holds for $k-1$, with some element $z^{(k-1)}\in \Gamma$. Due to Observation 1, there exists a $\Gamma$-conjugate $\tilde{z}$ of $z^{(k-1)}$ with $\tilde{z}_i \in W_j$ not fixed by the $i$-th component entry $(x_k)_i$ of $x_k$. Since conjugation in $\Gamma$ is defined componentwise, we have $\supp(\rho(z^{(k-1)}))=\supp(\rho(\tilde{z}))$.
	Let $z^{(k)}:=[\tilde{z}, x_k]$. Then $z^{(k)}_i \in W_j\setminus\{0\}$ by construction, and moreover $\rho(z^{(k)})$ is trivial on every component $m\ne i$ on which at least one of $\rho(\tilde{z})$ and $\rho(x_k)$ is not supported; in other words, due to the induction hypothesis, $\supp(\rho({z}^{(k)})) \subseteq (\bigcap_{m=1}^k \supp(\rho(x_m)))\setminus\{i\}$. This completes the proof of Claim 1.
	
	We next let $z=z^{(s)}\in \Gamma$ fulfill Claim 1 for $k=s$.
	By Claim 1, $\supp(\rho(z))=\emptyset$, i.e., all component entries of $z$ act trivially on $W$.
	Faithfulness of $U$ on $W$ thus implies that all component entries lie in $W$, i.e., $z\in \Delta$ and $z_i\in W_j\setminus\{0\}$.
	
	\underline{Claim 2}: By induction on $k=0,\dots, s$, there exists $y^{(k)}\in \Delta$ such that  $\{i\}\subseteq \supp(y^{(k)})\subseteq \bigcap_{m=1}^k\supp(\rho(x_m))$, and such that the $i$-th component $y^{(k)}_i$ of $y^{(k)}$ lies in $W_j\setminus\{0\}$.
	
	The base case $k=0$ is obtained with $y^{(0)}=z$ as chosen above. Assume inductively that there exists $y^{(k-1)}\in \Delta$ with $y^{(k-1)}_i\in W_j$ and $\{i\}\subseteq \supp(y^{(k-1)})\subseteq \bigcap_{m=1}^{k-1}\supp(\rho(x_m))$.
	Using again Observation 1), we obtain $\tilde{z}\in \Delta$ with $\supp(\tilde{z})=\supp(y^{(k-1)})$, and such that the $i$-th component $(\tilde{z})_i$ lies in $W_j$ and is not fixed by $\rho(x_k)_i\ne 1$.
	Finally, consider the commutator $y^{(k)}:=[\tilde{z},x_k]\in \Delta$.
	This has trivial $W_j$-part on all components not simultaneously supported by $\tilde{z}$ and $\rho(x_k)$, i.e., by construction of $\tilde{z}$, on all components outside $\bigcap_{m=1}^{k}\supp(\rho(x_m))$. Moreover, the $i$-th component is in $W_j\setminus\{0\}$ by construction, completing the induction.
	
	Letting $y=y^{(s)}\in \Delta$, we get $\supp(y)=\{i\}$. Conjugating $y$ by all $\gamma\in \Gamma$, we obtain that the set of elements of $\Delta$ supported only on the $i$-th component contains all of $W_j$.
	
	For the last assertion, the normality of $W^d$ is straightforward. For the minimality assertion, note that every non-zero vector $0\ne \gamma\in W'^d=W^d$ is admissible as a starting vector of the construction above, i.e., the normal subgroup generated by $\gamma$ is all of $W^d$.
\end{proof}

\begin{proof}[Proof of Lemma \ref{lem:normal} using Lemma \ref{lem:fullsocle}]
	In both cases 1) and 2) of Lemma \ref{lem:normal}, the claimed containment follows directly from Lemma \ref{lem:fullsocle} with $r=1$ and $W'=W$: in 1), take
	$C_p < H \le \AGL_1(p)$, with $W=C_p$ being the one-dimensional module under the action of $\{1\}\ne U'\le C_{p-1}$ (since $x_i$ are of order coprime to $q$). Similarly, in 2), take 
	$\ASL_2(2)=A_4\le H\le S_4=\AGL_2(2)$, with $W=\mathbb{F}_2^2$ and $C_3\leq U' \leq S_3 \cong S_4/W$ (since every $x_i$ is of order $3$). 
	Finally, the minimality follows since $W^d=W'^d=W_1^d$ is a minimal normal subgroup by Lemma \ref{lem:fullsocle}.  
\end{proof}

We collect some situations in which we wish to apply Lemma \ref{lem:normal} later on.
\begin{lem}
	\label{lem:support}
	In the setting of Lemma \ref{lem:normal}, the condition $|\bigcap_{j=1}^s\supp(x_j)| = 1$ for some $x_1,\dots, x_s\in K$ is in particular fulfilled in the following cases.
	\begin{itemize}
		\item[a)] Whenever $K/\soc(K)$ contains the augmentation subgroup ${\rm Aug}(C_2^d)$ (in case (1) of Lemma \ref{lem:normal}), resp. ${\rm Aug}(C_3^d)$ (in case (2)).
		\item[b)] Whenever there exists $x\in K$ (of the respective order specified in (1) and (2) of Lemma \ref{lem:normal}) with $1\le |\supp(x)|< d$, and the blocks image $A:={\rm Im}(G\to S_d)$ is primitive. In particular, the latter is automatic whenever $d\ge 3$ is a prime. 
		%\item[c)] Whenever $d\ge 11$ is a prime, $K$ is as in Lemma \ref{lem:normal}(1), and there exists an involution $x\in K$ whose support is of size $6$ and is invariant under some reflection in $D_d\supset C_d$.
	\end{itemize}
\end{lem}
\begin{proof}
	a) is a straightforward check. For b), it suffices to note the following: if $S\subseteq K$ is a maximal (with respect to inclusion) subset of elements of the respective specified order such that $\Delta:=\bigcap_{x\in S}\supp(x)\ne \emptyset$, then $\Delta$ is a block under the action of $A$. Indeed, suppose $\emptyset\ne \Delta\cap \Delta^g \ne \Delta$ for some $g\in A$; then $\bigcap_{x\in S\cup S^g}\supp(x)=\Delta\cap \Delta^g\ne \emptyset$, contradicting maximality of $S$.
\end{proof}

\subsubsection{On groups $G\le U\wr V$ with $V\in \{A_n,S_n\}$}
For groups $\Mon(g\circ h)$ with nonsolvable $\Mon(g)\in \{A_n,S_n\}$, large kernel conclusions are often more easily obtained. The following is immediate from \cite{Mor80}.
\begin{lem}\label{lem:invsub} Let $V=A_n,S_n$ for $n \geq 4$, and let $p$ be a prime. Then the only $V$-invariant subgroups of $C_p^n$ are $0$, ${\rm diag}(C_p^n)$, ${\rm Aug}(C_p^n)$ and $C_p^n$. 
\end{lem}

The following lemma is useful for analyzing the case of compositions $g\circ T_p$ where $g$ is an $S_n$-polynomial ($n\ge 4$).
%Below is what can be said without invoking polynomial monodromy (and without invoking Ritt theorems yet).
\begin{lem}\label{lem:snTp}
	Let $U\in \{C_p,D_p\}$ for a prime $p$, let $V\in \{A_n,S_n\}$  for $n\ge 4$,  and let $G\le U\wr V$ be a transitive subgroup projecting onto $V$, with block kernel $\Gamma:=\ker(G\to V)$. Then one of the following holds: 
	\begin{itemize}
		\item[i)] $\Gamma$ contains a subgroup $C_{p}^{n-1}$.
		\item[ii)] $G$ embeds into $U\times V$.
		\item[iii)] $n=4$, and $G/(\Gamma\cap C_p^4)$ (for $p>2$) resp.\ $G$ (for $p=2$) is isomorphic to $\SL_2(3) (\cong 2.A_4)$ or $\GL_2(3) (\cong 2.S_4)$. %\marginpar{There's also $8T40=2^3.GL_2(3)$ nonsplit, but that would go into case i).}
		\item[iv)] $n=6$, and $G$ is a nonsplit extension $G\cong 3.V$.
	\end{itemize}
\end{lem}
\begin{proof}
	Let  $H=G/(\Gamma\cap C_p^n)$ (for $p>2$), resp.\ $H=G$ (for $p=2$). Then $H$ embeds into $C_2\wr S_n=D_p\wr S_n/C_p^n$, and surjects onto $S_n$ or $A_n$. 
	We next consider the exact sequence 
	\begin{equation} 
		\label{splitseq}
		1\to H\cap C_2^n \to H\cap (C_2\wr A_n)\to A_n\to 1.\end{equation} 
	For $n\ge 5$, this sequence splits by \cite[Lemma 7.2]{KNR24}. For $n=4$, one may verify directly that there is exactly one nonsplit extension embedding into $C_2\wr A_4$, namely ${\rm SL}_2(3)$. 
	
	In total, exempting the case that $H$ contains $\SL_2(3)$ as a subgroup of index at most $2$ (i.e., Case iii) of the assertion), $\Gamma\cap C_p^n$ becomes a submodule of the $\mathbb{F}_p[A_n]$-permutation module. But the latter
	only has two nontrivial submodules, of dimensions $1$ and $n-1$ by Lemma \ref{lem:invsub}. To show that we are in Case i) of our assertion, we thus only need to show $|\Gamma\cap C_p^n|>p$. 
	Assume on the contrary that $|\Gamma\cap C_p^n|=p$. Then $\Gamma\in \{C_p,D_p\}$. 
	
	Moreover, $A_n$ acts trivially on $\Gamma$ since it has no nontrivial homomorphisms to $\mathbb F_p^\times$ for $n\ge 5$ (resp.\ since $A_4$ and $C_2=D_p/C_p$ have no nontrivial common quotient for $n=4$), see the map $\psi$ in Section \ref{sec:setup}.
	Now if the extension $\Gamma.A_n$ were nonsplit, then due to the splitting of \eqref{splitseq}, $C_p.A_n$ would have to be a nonsplit central extension. For $p=2$, a splitting is given above by the splitting of \eqref{splitseq}. For $p>2$,  such a nonsplit extension $C_p.A_n$ exists only for $p=3$ and $n=6,7$; the case $n=7$, however, cannot occur since a nonsplit extension $3.A_7$ does not have a faithful transitive degree $21$ action. Excluding thus finally Case iv) of the assertion, we obtain that $G$ contains $C_p\times A_n$
	as a subgroup of index dividing $4$. Since the normal subgroup $C_p$ has a complement in $C_p\times A_n$ and $[G:C_p\times A_n]$ is coprime to $p$, a famous theorem by Gasch\"utz \cite{Gasch52} asserts that $C_p$ has a complement $M$ in $G$. Since only $C_p$ and $A_n$ if $n\geq 5$ (resp.\ $V_4$ if $n=4$)  are the unique minimal normal subgroups of $G$, and only $A_n$ is contained in $M$, the image of $G$ acting on cosets of $M$ embeds into $U$ (i.e.\ to $C_p$ or $D_p$). Together with the restriction map $G\to V$, this yields a map $G\to U\times V$. Since $M\cap \Gamma$ is a subgroup of $\Gamma\le D_p$ not containing $C_p$, it contains no nontrivial normal subgroup of $G$, whence the induced map $G\to U\times V$ is an injection.
	
\end{proof}

\subsection{Ritt moves and direct products}

For this subsection, assume that $k=\overline{k}$ is an algebraically closed field of characteristic $0$.

\begin{prop}\label{prop:rittimpldiag}
	Suppose $g,h\in k[X]$ are indecomposable polynomials such that $g\circ h$ admits a Ritt move. Then $\mon(g\circ h)$ embeds into $\mon(g)\times \mon(h)$, with equality as long as $g\circ h$ is not linearly equivalent to $T_{pq}$ for primes $p,q$. %\AB{This lemma also works when $\mon(g)=S_4$.}
\end{prop}
\begin{proof}
	Following Corollary 2.11 and Theorem 2.13 of \cite{MZ09}, one has an equality $g\circ h = \tilde{h}\circ \tilde{g}$ with isomorphisms $\Mon(\tilde{g})\cong \Mon(g)$ and $\Mon(\tilde{h})\cong\Mon(h)$ of permutation groups. Denote by $U$ and $V$ the subgroups of $G:=\Mon(g\circ h)$ fixing a root of $\tilde{h}(X)-t$ and of $g(X)-t$ respectively. In particular the image of $V$ acting on cosets of $U\cap V$ equals (the image of $G$ acting on cosets of $U$, i.e.) $\Mon(h)$, and the image of $U$ acting on cosets of $U\cap V$ equals $\Mon(g)$. Since $U\cap V$ equals the stabilizer of a root of $g(h(X))-t$, its core $\textrm{core}_G(U\cap V)$ is trivial; on the other hand it is elementary that $G/\textrm{core}_G(U\cap V)$ injects into $G/\textrm{core}_G(U) \times G/\textrm{core}_G(V) = \Mon(h)\times \Mon(g)$. If additionally, $g\circ h$ is not linearly equivalent to $T_{pq}$, then by Ritt's theorems (see \eqref{eq:rittstep}), one of $g$ and $h$ (say, $h$, without loss of generality) is linearly equivalent to $X^p$. In particular, $\Mon(g\circ h)$ is then a subgroup of $\Mon(g)\times C_p$, surjecting onto $\Mon(g)$. If $\Mon(g\circ h)<\Mon(g)\times C_p$ was a proper subgroup, it would project isomorphically to $\Mon(g)$. This is however impossible since $\ker(\Mon(g\circ h)\to\Mon(g))>1$ due to the presence of a full cycle in $\Mon(g\circ h)$; Indeed, if $\sigma$ is an inertia group generator at $\infty$ -- in particular a cycle of length $\deg(g \circ h)$ -- then, by Lemma \ref{lem:fullcyindiag}, $\sigma^{\deg(g)}$ is a nontrivial element of the kernel. 
\end{proof}

\begin{prop}
	\label{prop:diagimplritt}
	Let $g,h\in k[X]$ be indecomposable polynomials of monodromy $U=\Mon(h)$ and $V=\Mon(g)$ satisfying the following:
	\begin{itemize}
		\item[a)] $U$ is cyclic or dihedral;
		\item[b)] $V$ is neither cyclic nor dihedral;
		\item[c)] $\Mon(g\circ h)$ embeds into $U\times V$.
	\end{itemize} 
	Then one of the following holds. 
	\begin{enumerate}[label=\textbf{\arabic*.}, ref=\arabic*]
		\item $g\circ h$ admits a Ritt move.
		\label{diagimplritt_rittcase}
		\item $h\sim X^p$, and moreover $$(p,\Mon(g))\in \{(2,\PSL_3(2)), (2,M_{11}), (2,\PSL_3(3)), (3, A_5), (3, \PGL_2(7))\}.$$ 
		\label{diagimplritt_exccase}
	\end{enumerate}
\end{prop}

\begin{proof}
	Let $G:=\Mon(g\circ h)$, and let $H<G$ be the stabilizer of a root of $g(X)-t$, so that the image of $G$ in the action on cosets of $H$ is the primitive group $V$. Let $V_1\le V$ be a point stabilizer, so that the projection $\pi:\Mon(g\circ h)\to \Mon(g)$ maps $H$ to $V_1$.
	Let $p$ be the prime such that $C_p\le U\le D_p$.
	We first claim that if $V_1$ does not have a composition factor $C_p$, then $g\circ h$ admits a Ritt move. To prove the claim, note that $V_1$ not having a composition factor $C_p$ means that $\pi^{-1}(V_1)=H$ has only a single composition factor $C_p$.
	This implies that there cannot be more than one conjugacy class of subgroups of $H$ such that the induced coset action of $H$ has image $C_p$ or $D_p$, for if there were two such subgroups $S_1,S_2\le H$, their intersection would necessarily be of index $p^2$ in $H$, yielding a quotient of $H$ (embedding into $D_p\wr D_p$) with more than one composition factor $C_p$, a contradiction.
	But on the other hand, the group $G$ has a subgroup $T$ with coset image $C_p$ or $D_p$, namely $T=G\cap (\{1\}\times V)$. Since $V$ is neither cyclic nor dihedral, the Galois group of $g(X)-t$ over the fixed field of $T$ must contain a nontrivial, and hence transitive normal subgroup of the primitive group $V$, and thus the image of $H$ in the action on cosets of $T\cap H$ is still $C_p$ or $D_p$. This means that the stabilizer $G_1<H$ of a root of $g(h(X))-t$ must be of this form $T\cap H$ and in particular has the two distinct nontrivial overgroups $H$ and $T$, implying that $g\circ h$ has a Ritt move, thus proving the claim.
	
	Since \cite{Mul2} provides an explicit short list of possibilities for $V$ other than symmetric or alternating groups of degree $\ge 4$, we are already reduced to the case $V=S_d$, $p=2$, and a short finite list of further cases $(p,V)$ (namely with $V_1$ having a composition factor $C_p$). To treat the former case, it is helpful to note that restriction from $G$ to $H$ yields an injection on the set of subgroups of index $2$. But since the abelianizations of $V=S_d$ and $V_1=S_{d-1}$ have the same $2$-rank, so do the abelianizations of $G$ and of $H$. Thus, restriction from $G$ to $H$ in fact yields a bijection on the set of index-$2$ subgroups, again implying that the point stabilizer $G_1<G$ has (an overgroup of index $2$ in $G$, and hence) two distinct nontrivial overgroups.
	
	Treating the latter case amounts to a short Magma computation, identifying subgroups $G\le D_p\times V$ (inside the finite list of candidate cases) acting as transitive subgroups of $D_p\wr V$ (more precisely, having a chain $G>H>G_1$ of subgroups with coset images $V$, and $C_p$ or $D_p$ respectively), but fixing only one nontrivial partition. Apart from the possibilities given in Case \ref{diagimplritt_exccase} (which can be verified with Magma to indeed occur as monodromy groups of a suitable polynomial with a {\it unique} decomposition), this leaves only one possibility, namely the group of transitive group label $15T21$ (isomorphic to $3.S_5$, embedding into $D_3\times S_5$ and having a unique maximal overgroup of the point stabilizer). This group, however, can be verified to have no generating genus-$0$ tuples corresponding to a polynomial ramification type. This concludes the proof.
\end{proof}

\begin{rmk}
	\label{rmk:diagimplritt}
	\begin{itemize}
		\item[a)]
		The final result of Theorem \ref{thm:mainres} shows that the conclusion of Proposition \ref{prop:diagimplritt}  holds unchanged even when dropping assumptions a) and b).
		\item[b)] Note that the assumption that $G:=\Mon(g\circ h)$ is in fact the geometric monodromy group of a polynomial was not used in the proof for the case $\Mon(g)=A_n,S_n$ for $n\ge 6$ (and for $n=5$, only to rule out the one further exceptional case $G=15T21$).
	\end{itemize}
\end{rmk}

\subsection{Topological ideas: coalescing of branch points and braid group action}
\label{sec:topol}
A useful idea to lower-bound the monodromy group of a polynomial $f$ with $r\ge 4$ branch points and branch cycles $\pi_1,\dots, \pi_r$ (with $\pi_1\cdots\pi_r=1$) is the following observation: assume that $\pi_1$ is the inertia group generator at infinity, i.e., is a full cycle in $\Mon_{\overline{k}}(f)$. Of course, 
$\langle \pi_1,\dots, \pi_{r-1}\rangle \supseteq \langle \pi_1,\dots, \pi_{r-2}\rangle$, where the left side equals $\Mon_{\overline{k}}(f)$, and the right side is still a transitive subgroup of $\Mon_{\overline{k}}(f)$ (due to the presence of the full cycle $\pi_1$);  the group $\langle\pi_1,\dots, \pi_{r-2}\rangle$ is thus the monodromy group of a polynomial\footnote{That the tuple $(\pi_1,\dots, \pi_{r-2}, \pi_{r-1}\pi_r)$ indeed yields the monodromy group of a {\it polynomial} can be seen easily from the fact that the full cycle $\pi_1$ is contained, together with the fact that the genus given by the tuple $(\pi_1,\dots, \pi_{r-2}, \pi_{r-1}\pi_r)$ is upper bounded by the one given by $(\pi_1,\dots, \pi_{r-2}, \pi_{r-1}, \pi_r)$ (which however is zero); the latter fact can easily be seen combinatorially from the Riemann-Hurwitz genus formula \eqref{rh_formula}, or from the topological interpretation via deformation of covers.} with branch cycles $\pi_1,\dots, \pi_{r-2}, \pi_{r-1}\pi_r$, which topologically corresponds to a deformation of a family of covers $\mathbb{P}^1\to \mathbb{P}^1$, letting the $(r-1)$-th and $r$-th branch point coalesce, see e.g., \cite{Cou96}. 

Note furthermore that the branch cycles of a polynomial map $f$ of degree $d$ and with $r$ branch points are only uniquely defined up to a relabeling of the points $1,\dots, d$ (i.e., conjugation in $S_d$) and action of the {\it Hurwitz braid group} $B_r$. The latter is defined as the group generated by braids $\beta_1,\dots, \beta_{r-1}$ fulfilling the relations 
$$ \beta_{i}\beta_j = \beta_j\beta_i, \text{ for all } i<j-1,$$
$$\beta_{i}\beta_{i+1}\beta_i = \beta_{i+1}\beta_{i}\beta_{i+1}, \text{ for all } 1\le i\le r-2, \text{ and }$$
$$\beta_1\beta_2\cdots \beta_{r-1}\beta_{r-1}\cdots\beta_1 = 1;$$ and acting on $r$-tuples of branch cycles via $$(\pi_1,\dots, \pi_r)^{\beta_i} := (\pi_1,\dots, \pi_{i-1}, \pi_i\pi_{i+1}\pi_{i}^{-1}, \pi_i,\dots, \pi_r).$$
See \cite[Chapter III.1]{MM} for the topological interpretation of this action.
In our applications, applying the braid group action can be useful for polynomials $f=g\circ h$ where $g$ is the generic type $S_4$-polynomial. 
%(i.e., with four branch points, of ramification type $[4]$, $[2,1^2]$, $[2,1^2]$, $[2,1^2]$). 
It is well-known, and  easy to  computationally verify in  degree $4$,  that the $4$-tuples $(\pi_1,\dots, \pi_4)$ of generic type as above lie in a single orbit of the braid group. Since the braid group action does not change the group generated by the respective $r$-tuple, we may therefore assume without loss of generality that the branch cycles for the generic type $S_4$-polynomial are $(1,2,3,4), (1,2), (1,3), (1,4)$ (with multiplication of permutations defined from right to left), and therefore letting either the second and third or the third and fourth branch point coalesce (in the sense explained above) yields an $S_4$-polynomial of special type.

\section{Solvable case: proofs of Theorems \ref{thm:affinesho}, \ref{thm:affS4I}, \ref{thm:affS4II} and \ref{thm:S4S4}}
Let $k$ be an algebraically closed field of characteristic $0$. We shall use implicitly Lemma \ref{lem:direct} and Remark \ref{rmk:readramel} several times when dealing with powers of inertia group generators.
\subsection{Proof of Theorem \ref{thm:affinesho}}
Suppose first that ${\rm Mon}(h)=C_p$. Let $u$ be the finite branch point of $h$. Assume we are not in the following cases:
\begin{itemize}
	\item ${\rm Mon}(g)=C_q$ and $u$ is a ramification point of $g$ -- equivalently, $g \circ h$ is linearly equivalent to $X^{pq}$;
	\item ${\rm Mon}(g)=D_q$, $p=2$, and $u$ is a special point of $g$ -- equivalently, $g \circ h$ is linearly equivalent to $T_{2q}$;
\end{itemize}
If $g(u)$ is a non-branch point of $g$, then the inertia group generator at $g(u)$ yields an element $x\in C_p^q$ of support of size $1$, which implies that $\Gamma = \langle C_q\cdot x\rangle=C_p^q$. Assume therefore from now on that $g(u)$ is a branch point of $g$ (in particular, due to the first excluded case above, this implies that $\Mon(g)=D_q$).
Squaring the inertia group generator at $g(u)$ as in Lemma \ref{lem:direct} yields an element  $x\in C_p^q$, commuting with a certain reflection of $D_q$ (namely, the projection to $D_q$ of the same inertia group generator), and with $|\supp(x)|\in \{1,2\}$. If $|\supp(x)|=1$, we have $\Gamma=C_p^q$. If $|\supp(x)|=2$, the above commuting implies $x={\bf e}_i+{\bf e}_j$ as an element of $\mathbb{F}_p^q$, and thus:
\begin{itemize}
	\item If $p$ is odd, then by Lemma \ref{lem:Szodd}, we obtain $\Gamma=C_p^q$.  
	\item If $p=2$, then by Lemma \ref{lem:Szodd}, we have ${\rm Aug}(C_2^q) \leq \Gamma$. However, by Lemma \ref{lem:fullcyindiag}, $(1,\dots,1)\in \Gamma \setminus {\rm Aug}(C_2^q)$, and hence $\Gamma=C_2^q$.  
\end{itemize}

Suppose next that ${\rm Mon}(h)=D_p$ for a prime $p \geq 3$. Then, unless $g \circ h$ is linearly equivalent to $T_{pq}$, the polynomial $g \circ h$ has a branch point, whose inertia generator, or its square, yields an element of $\Gamma/\soc(\Gamma)$ whose support is of size $1$ or $2$ or $4$, and is invariant under a reflection. In the first case, we have $\Gamma/\soc(\Gamma)=C_2^q$. In the second case,  $\Gamma/\soc(\Gamma)$ contains ${\rm Aug}(C_2^q)$ by Lemma \ref{lem:Szodd}. In the third case, $q \geq 5$ and by Lemma \ref{lem:ffkf_new},
$\Gamma/\soc(\Gamma)$ also contains ${\rm Aug}(C_2^q)$. Applying Lemmas \ref{lem:normal} and \ref{lem:support}, it follows that $\soc(\Gamma)=C_p^q$. Finally, in the case where $\Gamma/\soc(\Gamma)={\rm Aug}(C_2^q)$, it follows that $\Gamma=\{ (a_k)_{k=1}^q \in D_p^q \mid a_1 \cdots a_q \in C_p \}$ \footnote{For $n \geq 2$, let $\mathcal{L}_{p,n}=\{ (a_k)_{k=1}^n \in D_p^n \mid a_1 \cdots a_n \in C_p \}$.  Denote by $D_p^n \times_{C_2} D_p$ the fiber product along the canonical epimorphisms $D_p^n \twoheadrightarrow D_p^n/\mathcal{L}_{p,n}=C_2$   and $D_p \twoheadrightarrow D_p/C_p=C_2$. Then
	$\mathcal{L}_{p,n+1}\cong D_p^n \times_{C_2} D_p$.}.
\begin{rmk}\label{rem:more-cases}
	When $\Mon(h)=C_p$, the conclusion $C_p^q \leq \Gamma$ also holds in the following cases:
	\begin{itemize}\item $p,q$ are arbitrary integers, and $\Mon(g)=C_q$. The proof is the same as above.
		\item $p \geq 3$ is prime, $q$ is an odd integer, and $\Mon(g)=D_q$. Let $x \in C_p^q$ be as in the first part of the proof above. If $|{\rm supp}(x)|=1$, then $C_p^q \leq \Gamma$. If $|{\rm supp}(x)|=2$, then, by Lemma \ref{lem:Szodd}, we also deduce that $C_p^q \leq \Gamma$.
	\end{itemize}
\end{rmk}
\subsection{Proof of Theorem \ref{thm:affS4I}}
By replacing $g,h$ with $g\circ \mu,\mu^{-1}\circ h\circ\nu$, resp., for $\mu,\nu\in k[X]$ of degree $1$, we may assume $h=X^p$ or $T_p$.
\subsubsection{Assume $h=T_p$.}  
First, as soon as some inertia group generator $c\in \Mon(g\circ h)$ powers to an involution $c^m\in \Gamma$ supported on at most $3$ blocks, Lemma \ref{lem:havzero} (applied to the $\mathbb{F}_2[S_4]$ module $\Gamma/(C_p^4\cap \Gamma)$) yields that $\Gamma/(C_p^4\cap \Gamma)$ contains the augmentation subgroup. Lemma \ref{lem:normal}(1) and Lemma \ref{lem:support}a) then allow to conclude $\Gamma\supseteq C_p^4$.
The only situations in which the condition $|\supp(c^m)|\in \{1,2,3\}$ is not fulfilled for any branch point are the following.
\begin{itemize}
	\item[1)] When $g$ is the special $S_4$-polynomial and both branch points of $T_p$ are the preimages of the branch point of $g$ of ramification type $[3,1]$ (i.e., Case \ref{agls4_1_nonexc}\ref{agls4_1_nonexc2} of the theorem; indeed, here the inertia group generator $c$ fulfills that $c^3\in \Gamma$ has support of size $4$). 
	\item[2)] When both branch points of $T_p$ are special points over some branch point of $g$ of ramification type $[2,1,1]$.
\end{itemize}
Moreover, even in the above cases, one necessarily has  $\dim(\Gamma\cap C_p^4)\ge 3$ by Lemma \ref{lem:snTp} (applied with $n=4$, $U=\Mon(h)$, $V=\Mon(g)$ and $G=\Mon(g\circ h)$). Note here that the exceptional cases of Lemma \ref{lem:snTp} where $\Mon(g\circ h)$ embeds into $\Mon(g)\times \Mon(h)$, and where $\Mon(g\circ h)/(\Gamma\cap C_p^4)\cong \GL_2(3)$ are both impossible; indeed, the former would imply a Ritt move for $g\circ h$ by Proposition \ref{prop:diagimplritt} whereas a Ritt move involving both $T_p$ and an $S_4$-polynomial does not exist by Ritt's theorems; The latter would imply that all $4$-cycles of $\Mon(g)\cong S_4$ would lift to order $8$ elements modulo $(\Gamma\cap C_p^4)$ whereas the fact that the inertia group generator at infinity is a $4p$-cycle shows that this element remains of order $4$ modulo $(\Gamma\cap C_p^4)$.

The only remaining task to complete the proof of the theorem for $h=T_p$ is then to show that in Case 2) above, one has $\Gamma\supseteq C_p^4$ as long as the two finite branch points of $T_p$ are special points over two {\it different} $[2,1,1]$ branch points of $g$ (in particular, $g$ is necessarily the generic $S_4$ type here), since indeed the case when both lie over the {\it same} branch point is Case \ref{agls4_1_nonexc}\ref{agls4_1_nonexc3} of the theorem. 
To deal with this remaining case, we use the setup explained in Section \ref{sec:topol}. Namely, letting $\gamma_1,\dots, \gamma_4$ denote the branch points of $g\circ T_p$ and $\pi_1,\dots, \pi_4$ the respective branch cycles, we may assume without loss of generality that projection $\Mon(g\circ T_p)\to \Mon(g)\cong S_4$ yields the tuple $(\overline{\pi}_1,\dots, \overline{\pi}_4) = ((1,2,3,4), (1,2), (1,3), (1,4))$. Exactly two of the points $\gamma_2,\gamma_3,\gamma_4$ are extended by a branch point of $T_p$. Assume first that $\gamma_2$ is extended by a branch point of $T_p$. Then the triple $(\pi_1,\pi_2,\pi_3\pi_4)$ projects to a triple of ramification type $([4], [2,1^2], [3,1])$ in $S_4$. Moreover, $\langle \pi_1, \pi_2 \rangle$ is a transitive subgroup of $\Mon(g\circ T_p)\le D_p\wr S_4$, and is the monodromy group of a polynomial $\tilde{f}=\tilde{g}\circ \tilde{h}$, where $\tilde{g}$ is the special $S_4$-polynomial and $\Mon(\tilde{h})\le D_p$, see Section \ref{sec:topol}. Since the second inertia group generator is still $\pi_2$, this branch point is still extended by a branch point of ramification index $2$, whence necessarily $\Mon(\tilde{h})=D_p$. The other branch point of $\tilde{h}$ of ramification index $2$ must necessarily lie over the third branch point of $\tilde{f}$, whose inertia group generator however projects to cycle type $[3,1]$ in $S_4$. For this configuration, we have however already seen that $\Gamma\supseteq C_p^4$. The same must therefore hold in $\Mon(f)$. If $\gamma_2$ is {\it not} extended by a branch point of $T_p$, the analogous argument with the triple $(\pi_1,\pi_2\pi_3, \pi_4)$ gives the result. This concludes the case $h=T_p$.

\subsubsection{Assume $h=X^p$.} By Lemma \ref{lem:havzero}, one has $\dim(\Gamma\cap C_p^4)\ge 3$ as soon as there exists an inertia group generator $c\in \Mon(g\circ X^p)$ powering to an element $c^m\in\Gamma$ with $1\le |\supp(c^m)|\le 3$. Since the ramification indices at finite branch points of $g$ are all $2$ or $3$, the latter condition is certainly fulfilled by the inertia group generator at the unique finite branch point of $g\circ X^p$ with ramification index divisible by $p$, as soon as $p\ge 5$.
For $p\in \{2,3\}$, all possible configurations could in principle be calculated with Magma, but theoretical arguments are given below as well. For $p=3$, the only situation in which Lemma \ref{lem:havzero} is not directly applicable arises when the finite branch point $u=0$ of $X^3$ is a special point of $g$ lying over a branch point of ramification index $3$. This enforces that $g$ is the special $S_4$-polynomial, and then corresponds to Case \ref{agls4_1_exc1} of the theorem. For all $p\ne 2$, Lemma \ref{lem:fullcyindiag} moreover implies that as soon as $\dim(\Gamma\cap C_p^4)\ge  3$, one even has $\Gamma\supseteq C_p^4$, due to the containment of both the augmentation subgroup $\textrm{Aug}(C_p^4)$ and the diagonal submodule $\textrm{diag}(C_p^4)$. 

We are thus left with $h=X^2$.
Let $c$ be the inertia group generator at the branch point $g(0)$ of $g\circ X^2$, and again let $m$ be the smallest natural number such that $c^m\in \Gamma$. If $|\supp(c^m)|\in \{1,2,3\}$, then once again $\Gamma\supseteq \textrm{Aug}(C_2^4)\cong C_2^3$  by Lemma \ref{lem:havzero}; and if $\supp(c^m)\in \{1,3\}$, then even  $\Gamma=C_2^4$ (since $c^m$ is then an odd permutation, i.e., not contained in $\textrm{Aug}(C_2^4)$). The only cases left to consider are therefore when the branch point $0$ of $h=X^2$ extends a branch point of $g$ of ramification type $[2,1,1]$ (since indeed $c^m=1$ in case $0$ is a special point, and $|\supp(c^m)|=2$ in case $0$ is the ramification point over this $[2,1,1]$-branch point).
If $g$ is the generic type $S_4$-polynomial, a coalescing argument similar to the one carried out above reduces to the situation $\tilde{f}=\tilde{g}\circ X^2$ where $\tilde{g}$ is the special $S_4$-polynomial and the branch point $0$ of $X^2$ extends the $[3,1]$-branch point of $\tilde{g}$; here we already know $\Gamma= C_2^4$. There remains the case in which $g$ is a special $S_4$-polynomial and $0$ extends the unique $[2,1,1]$ branch point of $g$. The case where $0$ is a special point corresponds to Case \ref{agls4_1_exc2} of the theorem, and the monodromy group can directly be checked to equal ${\rm GL}_2(3)$ with Magma. The case where $0$ is a ramification point corresponds to Case \ref{agls4_1_nonexc}\ref{agls4_1_nonexc1} of the theorem, and Lemma \ref{lem:havzero} yields $\Gamma\supseteq \textrm{Aug}(C_2^4)$.

\subsection{Proof of Theorem \ref{thm:affS4II}}
As in the proof of Theorem \ref{thm:affS4I}, we replace $g,h$ by their composition with linear polynomials to assume $g=X^p$ or $T_p$.  Note first that, in order to prove $\Gamma=S_4^p$, it suffices to prove $\Gamma/(\Gamma\cap A_4^p)=C_2^p$. Indeed, due to Lemma \ref{lem:normal}(1) and Lemma \ref{lem:support}(a), the latter condition implies $\Gamma/(\Gamma\cap V_4^p) = S_3^p$; and analogously, this condition, due to Lemma \ref{lem:normal}(2) and Lemma \ref{lem:support}(a), implies $\Gamma\supseteq V_4^p$. 

Note furthermore that $h$ has at least one finite branch point $u\in k$ whose inertia group is generated by a transposition in $\Mon(h)$. Let $c\in \Mon(X^p\circ h)$ (resp., $\Mon(T_p\circ h)$) be the inertia group generator at $u^p$ (resp., at $T_p(u)$), and let $m\in \mathbb{N}$ be minimal such that $c^m$ is contained in $\Gamma$ and of order $\le 2$. As soon as $c^m$ is a transposition, we know that $\Gamma/(\Gamma\cap A_4^p) = C_2^p$ and hence $\Gamma=S_4^p$ by Lemma \ref{lem:support}.
If furthermore $p>2$, then $\Gamma/(\Gamma\cap A_4^p) = C_2^p$ and hence $\Gamma=S_4^p$ follows similarly  as soon as $c^m$ is a double transposition supported on two blocks: Indeed,  $\Gamma/(\Gamma\cap A_4^p)$  then contains $\textrm{Aug}(C_2^p)$ by Lemma \ref{lem:Szodd}, but also contains the diagonal by Lemma \ref{lem:fullcyindiag}, which is generated by the $p$-th power of the inertia group generator at $\infty$. This leaves us with only the following cases to consider:
\begin{itemize}
	\item[Case 1.] $p\ge 3$, $g=X^p$. In this case the only situations left to consider are where 
	either no transposition branch point of $h$ lies over a non-branch point of $g$, or 
	(all) three such points lie over the same non-branch point of $g$. 
	\begin{itemize}[leftmargin=*]
		\item[Case 1.1.] The former case is only possible when $h$ is a special $S_4$-polynomial and the transposition branch point $u$ of $h$ equals the ramification point $u=0$ of $g$. Since in this case, the unique branch point of $h$ of ramification type $[3,1]$ lies over a non-branch point of $X^p$, the respective inertia group generator is an element $x$ of $\Gamma\cap A_4^p$ supported on only one block, from which it follows that $\Gamma\supseteq \langle C_p\cdot x\rangle=A_4^p$. This concludes (the case $p\ge 3$ of) Case \ref{agls4_2_nonexc}\ref{agls4_2_nonexc1} of the theorem.
		\item[Case 1.2.] The latter case implies that $h$ is a generic $S_4$-polynomial with three transposition branch points $u,v,w$ fulfilling $u^p=v^p=w^p$. For $p=3$,  one gets monodromy group $S_4\times C_3$ by a direct computation (e.g.\ with Magma). This corresponds to the ``Ritt move" Case \ref{agls4_2_exc} of the theorem. Assume from now on $p\ge 5$. Here, since we have an inertia group generator $c\in \Gamma$ which is a triple transposition supported on exactly three blocks, it follows from Lemma \ref{lem:normal}(1) (applied with the quotient $\Gamma/(\Gamma\cap V_4^p)\le S_3^p$ in the role of $K$) and Lemma \ref{lem:support}b) that $C_3^p\le \Gamma$ and hence $A_4^p\le \Gamma$.\footnote{Note that there are indeed cases here where $[S_4^p:\Gamma]$ is arbitrarily large. E.g., whenever $p=2^q-1$ is a Mersenne prime, a suitable choice of the preimages of $u^p$ leads to the projection of the corresponding inertia group generator to $S_4^p/A_4^p=C_2^p$ generating the (cyclic!) binary Hamming code of length $2^q-1$, which has dimension $2^q-q-1$. A suitable choice of branch points therefore leads to a case with $[S_4^p:\Gamma]=2^q$.} This concludes Case \ref{agls4_2_nonexc}\ref{agls4_2_nonexc3} of the theorem.
	\end{itemize}
	\item[Case 2.] $p=2$, i.e., $g=X^2$; and moreover no inertia group generator of $g\circ h$ powers to a transposition. This is only possible if one branch point $u$ of $h$ with transposition inertia is the ramification point $u=0$ of $g$; and moreover $h$ is either the special $S_4$-polynomial, or $h$ is generic with two more finite branch points $v,w$ mapping to the same point under $g=X^2$. In both cases, the group structure can be verified directly with Magma. The former case behaves as in the case $p\ge 3$, and together with it, forms Case \ref{agls4_2_nonexc}\ref{agls4_2_nonexc1} of the theorem, whereas the latter case is Case \ref{agls4_2_nonexc}\ref{agls4_2_nonexc6} of the theorem.
	\item[Case 3.] $p\ge 3$, $g=T_p$.
	In this case, we can additionally use Lemma \ref{lem:ffkf_new}, which yields that $\Gamma/(\Gamma\cap A_4^p)\supseteq \textrm{Aug}(C_2^p)$ (and hence $\Gamma=S_4^p$ by the same argument as at the beginning of the proof) as soon as $\Gamma$ contains a quadruple transposition supported on four blocks which are invariant as a set under a suitable reflection of $D_p$. We distinguish further between the cases ``$h$ special" and ``$h$ generic" $S_4$-polynomials.
	\begin{itemize}[leftmargin=*]
		\item[Case 3.1.] $h$ is special. Here the only case left to consider is when the unique transposition branch point $u$ of $h$ is a special point of $T_p$ (i.e., $u=\pm 2$). In this case, the branch point of $h$ of ramification type $[3,1]$ is either an unramified point of $T_p$, or a ramification point of ramification index $2$; if $c\in \Mon(T_p\circ h)$ denotes the respective inertia group generator, it follows in both cases that $c^2\in \Gamma\cap A_4^p$ is an element of order $3$ supported on at most two blocks. By Lemma \ref{lem:Szpee}, this implies that the $\mathbb{F}_3[C_p]$ module $(\Gamma\cap A_4^p)/(\Gamma \cap V_4^p)$ contains $\textrm{Aug}(C_3^p)$. By Lemma \ref{lem:normal}(2) and Lemma \ref{lem:support}a), one has $\Gamma\supseteq V_4^p$. This concludes Case \ref{agls4_2_nonexc}\ref{agls4_2_nonexc2} of the theorem.
		\item[Case 3.2.] $h$ is generic. Here the only cases left to consider are the following: 
		\begin{itemize}[leftmargin=*]
			\item[Case 3.2.1.] In case at least one of the three finite branch points $u,v,w$ of $h$ lies over a non-branch point of $T_p$, all three of them must lie over the same non-branch point, i.e., $T_p(u)=T_p(v)=T_p(w)\notin\{\pm 2\}$. For $p=3$, the group structure may again be directly verified with Magma, yielding Case \ref{agls4_2_nonexc}\ref{agls4_2_nonexc5} of the theorem. Hence, assume $p\ge 5$. Here, it follows from Corollary \ref{cor:dp_smallweight} that $\Gamma/(A_4^p\cap \Gamma) = C_2^p$.
			\item[Case 3.2.2.] In case all of $u,v,w$ lie over branch points of $T_p$, two of them, say $u,v$, must be ramification points over the same branch point of $T_p$.
			If, however, $w$ is not also a ramification point over the same branch point, then the respective inertia group generator $c\in \Mon(T_p\circ h)$ fulfills that $c^2\in \Gamma$ is an involution supported on exactly four blocks, invariant as a set under a reflection of $D_p$ (namely the projection of $c$ to $D_p$). As remarked above, this case would yield $\Gamma=S_4^p$, so that we are after all left with the case that all of $u,v,w$ are ramification points over the same branch point of $T_p$ (i.e., Case \ref{agls4_2_nonexc}\ref{agls4_2_nonexc4} of the theorem). In this case, $c^2\in \Gamma$ is an involution whose support is of size $6$ (which has to be less than $p$). 
			It then follows from Lemma \ref{lem:normal}(1) (applied with the quotient $\Gamma/(\Gamma\cap V_4^p)\le S_3^p$ in the role of $K$) and Lemma \ref{lem:support}b) that $\Gamma\supseteq A_4^p$.
		\end{itemize}
	\end{itemize}
	
\end{itemize}

\subsection{Proof of Theorem \ref{thm:S4S4}}
Since $h$ has at most three finite branch points (all of ramification index $2$ or $3$), it follows that, if there is at least one branch point of $g\circ h$ which is a non branch point of $g$, then either $\Gamma/(\Gamma \cap A_4^4)$ contains an involution (namely, the third power of an inertia group generator) of support $\le 2$, or $\Gamma\cap A_4^4$ contains an element of order $3$ and support $1$. From this, it follows quickly that $\Gamma\supseteq A_4^4$, whence we may restrict to the case that every branch point of $g\circ h$ is a branch point of $g$, and in particular $g\circ h$ has at most four branch points.

Now the assertion is most conveniently checked using a direct Magma search for genus zero $3$- and $4$-tuples inside $S_4\wr S_4\le S_{16}$.

\section{Nonsolvable case: proofs of Theorems \ref{thm:arbnons}, %\ref{thm:nonsnons}, 
	\ref{thm:nonsaffI}, %\ref{thm:nonsaffII}
	and \ref{thm:nonsS4I}} %and \ref{thm:nonsS4II}}
	\label{sec:moncomplength2}
	Let $k$ be an algebraically closed field of characteristic $0$.
	
	In the following, we call an indecomposable polynomial $g\in k[X]$ (as well as its monodromy group $\Mon(g)$) {\it exceptional}, if its monodromy group is nonsolvable, but does not contain the alternating group. Recall that the exceptional indecomposable polynomials are known due to \cite{Mul2}, and their monodromy groups belong to a short finite list, cf.\ Theorem \ref{prop:primmon}.

	\subsection{Proof of Theorem \ref{thm:arbnons}}
	This follows from \cite[Corollary 4.4]{KNR24}, except in the case when
	there is a Ritt move for $g \circ h$, in which case necessarily $g=X^p$ up to linear equivalence, by Ritt's theorems. 
	Finally, the assertion $\Mon(g\circ h)\cong \Mon(g)\times \Mon(h)$ for the latter case follows directly from Proposition \ref{prop:rittimpldiag}.

	\subsection{Proof of Theorem \ref{thm:nonsaffI}} Set $n:=\deg(g)$. We replace $g,h$ by their compositions with linear polynomials to assume $h=X^p$ or $T_p$. 
	\subsubsection{Assume $\mon(g)=S_{n}$ or $\mon(g)=A_{n}$.}
	The case where there is a Ritt move for $g \circ h$ is covered by Theorem \ref{thm:arbnons}. In the following, we may assume there is no Ritt move for $g \circ h$. 
	Assume that $\Gamma$ does not contain a subgroup $C_p^{n-1}$. Then
	by Lemma \ref{lem:snTp}, either $(p,n)=(3,6)$ or
	$\Mon(g\circ h)$ embeds into $\Mon(h)\times \Mon(g)$. The former case can in fact be excluded since the nonsplit extension $3.S_6\le S_{18}$ does not contain an $18$-cycle, and hence cannot be the monodromy group of a polynomial.
	
	Since we have however excluded the existence of a Ritt move, Proposition \ref{prop:diagimplritt} 
	yields $p=3$, $h=X^3$ and $\Mon(g)\cong A_5$ as the only possibility; cf.\  Table \ref{table:sporcases}.
	
	\subsubsection{Assume $\Mon(g)$ is exceptional} 
	Assume first that $h=X^p$. Since the list of polynomial ramification type of exceptional nonsolvable monodromy groups is known explicitly (cf.\ \cite{Mul2}), the verification for $p=2,3$ amounts to a Magma calculation. (The calculation runs through genus-$0$ tuples of elements in $C_p\wr \Mon(g)$ that project onto one of the possible tuples in $\Mon(g)$, and keeps note of the subgroups generated by such tuples). Assume henceforth that $p\ge 5$. 
	The precise list of exceptional nonsolvable ramification types (in particular, the fact that, for all exceptional polynomial monodromy tuples, the ramification indices at  finite branch points  have only the prime divisors $2$ and $3$) yields the existence of an element of order $p$ in $\ker(\Mon(g\circ h)\to \Mon(g))$ whose support (as an element of $C_p^{n}$) is strictly smaller than $\{1,\dots, n\}$; more specifically, to obtain such an element, one takes an appropriate power, using Lemma \ref{lem:direct} and Remark \ref{rmk:readramel}, of the unique inertia group generator at a finite branch point of $g\circ h$ with ramification index divisible by $p$. But by \cite{Mor80}, for $p\ge 5$, the only nontrivial submodules of the $\mathbb{F}_p$ permutation module of the simple group $\soc(\Mon(g))$ are the diagonal and augmentation. Since we have already identified a nondiagonal element of order $p$, we obtain that $\ker(\Mon(g\circ h)\to \Mon(g))$ contains augmentation, as claimed.
	
	Assume now that $h=T_p$, $p\ge 3$. 
	The case $p=3$ is dealt with by a direct Magma computation, see the proof of Theorem \ref{thm:nonsS4I}, and yields exactly one exceptional monodromy group $\Mon(g\circ h)=\textrm{39T248}$ (with $\Mon(g)=\PSL_3(3)$ and further specifications as given in Table \ref{table:sporcases}). Therefore, let $p\ge 5$ from now on. 
	Assume on the contrary that $|\Gamma \cap C_{p}^n| < p^{n-1}$. Set $\tilde{G}:=\Mon(g\circ h)/(\Gamma\cap C_p^n) \le C_2\wr \Mon(g)$. Then $\tilde{\Gamma}:=\ker(\tilde{G}\to \Mon(g)) =\Gamma/(\Gamma \cap C_{p}^n)$ is of size at most $|\tilde \Gamma| \le 2$ by Lemma \ref{lem:normal}(1) and Lemma \ref{lem:support}b), since $\tilde \Gamma$ is invariant under the conjugation action of  $\Mon(g)$ (which is primitive on $\{1,\ldots,n\}$). Note that $\tilde G$ is the monodromy group of $g\circ \psi$, where $\psi$ is the degree-$2$ rational function\footnote{Namely, ramified exactly at the two finite branch points $\pm 2$ of $T_p$, and unramified at $\infty$.} parameterizing the unique quadratic subextension of the splitting field of $T_p(X)-t$.  A Magma computation,\footnote{One may assume for the computation that the set of branch points of $g\circ \psi$ is the same as that of $g$, since putting one or two branch points of $\psi$ over a non-branch point of $g$ results in an element of $\tilde{\Gamma}$ with support of size $1$ or $2$, yielding $|\tilde{\Gamma}|\ge 2^{n-1}$, as seen many times before in analogous situations.} 
	using the list of exceptional ramification types from \cite{Mul2}, now yields the following conclusion: %  intermediate result:

	Whenever $|\tilde{\Gamma}| \leq 2$, the extension $\tilde{\Gamma}.\Mon(g)$ is split.

	In total, we have therefore reduced to the case that the extension $\tilde{\Gamma}.\Mon(g)$ is split. This means that $\Gamma\cap C_p^n$ becomes a submodule of a module ${\rm Ind}_H^{\Mon(g)} \chi$, where $H\le \Mon(g)$ is a point stabilizer and $\chi$ is either the trivial character or a quadratic character of $H$, cf.\ Section \ref{sec:setup}. In the former case, ${\rm Ind}_H^{\Mon(g)} \chi$ is the permutation module, and since $p\ge 5$, it follows from \cite{Mor80} that $\dim(\Gamma\cap C_p^n) \in \{1,n-1,n\}$. The same assertion holds for quadratic $\chi$  by a direct Magma verification. More precisely, for all $p \divides |\Mon(g)|$ with $p\geq 5$, we  explicitly construct and decompose the induced modules. To do so  for $p\not\divides |\Mon(g)|$, we compute the corresponding complex characters which coincide with those over $\oline{\mathbb F}_p$ by \cite[Prop.\ 4.4.3]{Ser77}.
	To conclude the proof, we therefore only need to exclude the case $|\Gamma\cap C_p^n|=p$. In that case, $\Gamma$ is diagonal, i.e., $\Gamma=D_p$ or $\Gamma=C_p$.
	
	We therefore either have a) $\Mon(f) = C_p.(C_2.\Mon(g))$, where the central subgroup $C_2$ of $C_2.\Mon(g)$ acts nontrivially on $C_p$; or b) $\Mon(f)=C_p.\Mon(g)$.  We claim that in both cases  $\Mon(f)\le D_p\times \Mon(g)$, from which Proposition \ref{prop:diagimplritt} implies that $g\circ T_p$ admits a Ritt move. The latter is however impossible by Ritt's theorems.
	
	To prove the claim, set $G=\Mon(g\circ h)$, and let $H\leq G$ be a block-stabilizer in the action through $\Mon(g)$ (i.e.\ $H\leq G$ is a preimage of a point-stabilizer in $\Mon(g)$). 
	Let $N$ denote the kernel of the conjugation action $G\to \Aut(C_p)$. The image of the action of $N$ on $G/H$ is nontrivial
	since $\Mon(g)$ is nonsolvable but $G/N\leq \Aut(C_p)$ is abelian; hence this image is a nontrivial normal subgroup of the primitive group $\Mon(g)$ and hence transitive by \cite[Thm.\ 1.6A]{DM96}. Thus, $N$ acts transitively on the blocks in $G/H$.
	It follows that $G=HN$ and hence the image of the above conjugation action is $G/N\cong H/H\cap N\cong C_2$.
	
	In a) the fact that $\oline N=\ker(C_2.\Mon(g)\to \textrm{Aut}(C_p))=N/C_p$ does not contain the central subgroup $C_2$ enforces the extension $C_2.\Mon(g)$ to split; indeed, $C_2$ is disjoint from $\oline N$ and the above shows $C_2.\Mon(g)/\oline N\cong C_2$. Moreover the central extension $C_p.\oline{N}\cong C_p.\Mon(g)$
	is also split since none of the simple groups $\soc(\Mon(g))$ has a  Schur multiplier of order divisible by a prime divisor $p>3$. This gives $\Mon(f) = D_p\times \Mon(g)$.
	
	Similarly, in b), we obtain $G=C_p\rtimes \Mon(g)$, with $\Mon(g)$ acting on $C_p$ either trivially or via $C_2\le \textrm{Aut}(C_p)$. In this last case, $N=\ker(G\to\Aut(C_p))$ maps to  an index-$2$ normal subgroup $\oline N\le \Mon(g)$ such that $G=(C_p\times \oline N).C_2$, with the outer $C_2$ acting diagonally on $C_p$ and $\oline N$, which yields the desired  embedding into $D_p\times \Mon(g)$. This concludes the proof.
	
	\subsection{Proof of Theorem \ref{thm:nonsS4I}}
	We begin by noting that the cubic subextension of the Galois closure of $h(X)-t$, with an $S_4$-polynomial $h$, is always given by $\varphi(x)-t$ with a rational function $\varphi$; namely of ramification type $([3],[2,1],[2,1])$ for the special $S_4$-polynomial $h$, and  of ramification type $([2,1],[2,1],[2,1],[2,1])$ for the generic $S_4$-polynomial; indeed, this follows by projecting branch cycles in $S_4$ under the projection $S_4\to S_4/V_4\cong S_3$.
	We now consider the quotient of $\Mon(g\circ h)$ by the normal subgroup $V_4^{\Mon(g)}\cap\Gamma$. This is the monodromy group of $g\circ \varphi$, with the above rational function $\varphi\in k(x)$.
	
	First, assume $\Mon(g)\in \{A_n,S_n\}$, $n=\deg(g)$. Note that the exceptional case $\Mon(g\circ \varphi)\le 3.S_6$ of Lemma \ref{lem:snTp} cannot occur, e.g., since in the nonsplit extension $3.S_6\le S_{18}$ the $6$-cycles of $S_6$ do not lift to elements of order $12$, which would however have to be the case for the inertia group generator at infinity.  Lemma \ref{lem:snTp} therefore gives that either $\ker(\Mon(g\circ \varphi)\to \Mon(g))\supseteq C_{3}^{n-1}$ (in which case there is nothing to prove), or $\Mon(g\circ \varphi)\le \Mon(g)\times S_3$.
	Proposition \ref{prop:diagimplritt} (with Remark \ref{rmk:diagimplritt}b), since $\varphi$ is not a polynomial\footnote{The one exceptional case $G=15T21$  from Remark \ref{rmk:diagimplritt} can be excluded ad hoc; e.g., this group has no element of cycle structure $[10.5]$, which would however be needed as the inertia group generator at infinity.}) now implies that $g\circ \varphi$ has a non-unique decomposition, i.e., $g\circ \varphi = \tilde{\varphi}\circ \tilde{g}$, for some degree-$3$ function $\tilde{\varphi}$. Note that since $\infty$ has only two preimages (poles) under $\varphi$, it also has only two preimages under $g \circ \varphi$ and hence is equivalent to a Laurent polynomial. 
	It now can either be verified directly (considering ramification at infinity), or by invoking \cite[Theorem 1.1]{Pak09} that a Laurent polynomial with a non-unique decomposition cannot be the composition of functions with $S_3$- and nonsolvable monodromy.
	
	Next assume that $\Mon(g)$ is exceptional. This case can in principle be dealt with by exhaustive computation inside the wreath products $S_4\wr \Mon(g)$, since $\deg(g\circ h)$ is absolutely bounded by $31\cdot 4$; however, brute force computations inside the larger degree groups seem difficult, and we therefore give some useful simplifications. Again, consider the $3$-part $\tilde{\Gamma}$ of $\ker(\Mon(g\circ \varphi)\to\Mon(g))$ as above. For the exceptional $g$ with $\deg(g)>15$, i.e., $\Mon(g)\in \{\PGaL_3(4),M_{23},\PSL_5(2)\}$, all extensions $C_2^k.\Mon(g)\le C_2\wr \Mon(g)$ are split, whence $\tilde{\Gamma}$ becomes an $\mathbb{F}_3[\Mon(g)]$-module, and in particular a submodule of the permutation module under the simple group $\soc(\Mon(g))$. 
	%Note to self: for those simple groups the point stabilizer is a perfect group, thus justifying ``permutation module", see Section 3.2.
	But due to 
	\cite{Mor80}, either $|\tilde{\Gamma}|\ge 3^{\deg(g)-1}$ or $|\tilde{\Gamma}|=3$. We therefore only need to deal with the case of a diagonal extension $\Mon(g\circ \varphi)=S_3.\Mon(g)\le S_{3\deg(g)}$. In all cases, the only such group is the direct product, with a nonunique decomposition, which is impossible, as in the case $\Mon(g)=A_n,S_n$.
	%Now for M_{23} and $PSL_5(2)$ immediately clear that $\Mon(g\circ \varphi)\le S_3\times \Mon(g)$ (hence Ritt move, hence impossibility), but for P\Gamma L_3(4), Schur multiplier is div. by 3...
	
	For $\deg(g)\le 15$, the group $\Mon(g\circ \varphi)$ is in the realm of Magma's transitive group database, and we can directly search for cases of transitive subgroups of $S_3\wr \Mon(g)$ projecting onto $\Mon(g)$ and with block kernel acting on a block as $S_3$ such that there is (no Ritt move, i.e.) a unique maximal block system and such that $|\tilde{\Gamma}| < 3^{\deg(g)-1}$. This yields exactly the candidate groups \texttt{TransitiveGroup}$(k,j)$ for 
	\begin{align*}
(k,j)\in \{ & (39,134),(39,206),(39,218),(39,248),(33,60),  \\ & (33,69),(33,83),(33,97),(30,1650),(30,2251)\}
\end{align*} 
and now inside these a direct search for generating genus zero tuples leaves only $39T206$, $39T248$ (both with $\Mon(g)=\PSL_3(3)$) and $33T60$ (with $\Mon(g)=\PSL_2(11)$). 
Since in all cases, the $3$-part of $\Gamma$ is of order $>3$ and $\Gamma$ is invariant under the primitive group $\Mon(g)$, it follows from Lemma \ref{lem:normal}(2) with Lemma \ref{lem:support}b) that $\Gamma\supseteq V_4^{\Mon(g)}$.

\section{Higher compositions}
\label{sec:highercomp}
\subsection{Composition of three and more Chebyshev polynomials}
Let $k$ be a field of characteristic $0$. In this section we demonstrate some ways in which the results and techniques of this paper can be used to obtain lower bounds for kernels also in case of composition of more than two indecomposable polynomials.

The following theorem shows that arbitrarily long compositions of polynomials linearly equivalent to Chebyshev polynomials always have large kernel in the absence of Ritt moves.
\begin{thm}
\label{thm:comp_dihedral}
Let $f_1,\dots, f_r\in k[X]$ be polynomials of odd prime degrees $p_1,\dots, p_r$, $r\ge 2$, with $\Mon_{\overline{k}}(f_i) = D_{p_i}$, $i=1,\dots, r$, and such that $f_{i}\circ f_{i+1}\not\sim T_{p_{i}p_{i+1}}$ for all $i=1,\dots, r-1$. Set $h=f_1\circ\dots\circ f_{r-1}$, $f=h\circ f_r$, and $\Gamma=\ker(\Mon(f)\to \Mon(h))$. Then the following hold:
\begin{itemize}
	\item[a)] $\Gamma\supseteq C_{p_r}^{\deg(h)}$.
	\item[b)] $\Gamma/C_{p^r}^{\deg(h)} \supseteq C_2^{\deg(h)\cdot (1-\frac{1}{p_{r-1}})}$.
\end{itemize}
\end{thm}
\begin{proof}
By induction over $r$. It suffices to prove the assertion for algebraically closed fields $k=\overline{k}$. The case $r=2$ follows from Theorem \ref{thm:affinesho}. Assume that the assertion holds for $r-1\ge 2$. In order to control the $p_r$-part of $\Gamma$, we first estimate the $2$-part. For this, let $h_2=f_1\circ\cdots\circ f_{r-2}$,  $K=\Gamma\cap C_{p_r}^{\deg(h)}$, $N=\ker(\Mon(f)\to \Mon(h_2))$ and $G:=\Mon(f)/K$. Then $G = \Mon(h\circ \varphi_r)$, where $\varphi_r\in k(X)$ denotes the quadratic rational function corresponding to the unique quadratic subextension of the Galois closure of $f_{r}(X)-t$. In particular, $G$ embeds into $H\wr \Mon(h_2)$, where $H:=\Mon(f_{r-1}\circ \varphi_r)$.

By Theorem \ref{thm:affinesho} (applied to $f_{r-1}\circ f_r$), we have ${\rm Aug}(C_2^{p_{r-1}}).D_{p_{r-1}}  \le H \le C_2\wr D_{p_{r-1}}$. We next claim that the image $\tilde{H}$ of the blocks kernel $\tilde{\Gamma}:=\ker(\Mon(h\circ \varphi_r)\to \Mon(h_2)) \cong N/K$ under projection to any of the $\deg(h_2)$ blocks still contains ${\rm Aug}(C_2^{p_{r-1}})$. For this, note that $\tilde{H}$ is a normal subgroup of $H$ and contains an element of order $p_{r-1}$ as is evident from ramification over infinity. But then $\tilde{H}$ contains the normal closure of a $p_{r-1}$-Sylow group of $H$, i.e., contains all its elements of order $p_{r-1}$. Since ${\rm Aug}(C_2^{p_{r-1}}).C_{p_{r-1}} = (C_2\wr C_{p_{r-1}})\cap \textrm{Alt}(2p_{r-1})$ has no element of order $2p_{r-1}$, all elements in ${\rm Aug}(C_2^{p_{r-1}}).C_{p_{r-1}}$ outside of ${\rm Aug}(C_2^{p_{r-1}})$ are of order $p_{r-1}$, whence $\tilde{H}\supseteq {\rm Aug}(C_2^{p_{r-1}}).C_{p_{r-1}}$, showing the claim.

We will now choose a setup allowing us to invoke Lemma \ref{lem:fullsocle}. Let $\widehat{\sigma}$ be a generator of a cyclic transitive subgroup in $\Mon(h_2)$ (e.g., an inertia group generator at infinity), and $\sigma$ a preimage of $\widehat{\sigma}$ in $G$. Set $\tilde{G}:=\langle \tilde{\Gamma},\sigma\rangle$. Then $\tilde{G}\le \tilde{H}\wr C_d$ for $d:=\deg(h_2)$, and $\tilde{\Gamma}$ is still the blocks kernel in $\tilde{G}$. Moreover, there is a natural map $\rho: \tilde{\Gamma}\to D_{p_{r-1}}^d=\Mon(f_{r-1})^d$ whose image equals the kernel $\ker(\Mon(h)\to \Mon(h_2))$. By the induction hypothesis, $\rho(\tilde{\Gamma})$ contains $C_{p_{r-1}}^d$. Also note that ${\rm Aug}(C_2^{p_{r-1}})$ is  a direct sum of faithful irreducible modules since the only nontrivial normal subgroup $C_{p_{r-1}}\le D_{p_{r-1}}$ acts faithfully on each nondiagonal submodule of $C_2^{p_{r-1}}$. 
Now, applying Lemma \ref{lem:fullsocle} (to $\tilde{G}$, $\tilde{\Gamma}$, $\tilde{H}$), gives ${\rm Aug}(C_2^{p_{r-1}})^d \subseteq \ker(\rho) = \ker(\Mon(h\circ \varphi_r)\to\Mon(h)) = \Gamma/(\Gamma\cap C_{p^r})$, which gives the order of the $2$-part asserted in b).

From this, a) will follow by yet another application of Lemma \ref{lem:fullsocle}: this time, take more simply $\tilde{\Gamma}=\Gamma$ and $\tilde{G}=\langle\Gamma, \sigma\rangle$ for a cyclic transitive subgroup $\langle\sigma \rangle$ of $\Mon(f)$. Then $\tilde{G}\le D_{p_r}\wr C_{d'}$ for $d':=\deg(h)$. Setting $\rho:D_{p_r}^{d'}\to C_2^{d'}$, we note that we have just shown $\rho(D_{p_r}^{d'})$ to contain ${\rm Aug}(C_2^{p_{r-1}})^{d'/p_{r-1}}$. Since $p_{r-1}\ge 3$, the group ${\rm Aug}(C_2^{p_{r-1}})$ contains two elements whose supports have just one element in common. The same then of course holds for ${\rm Aug}(C_2^{p_{r-1}})^{d'/p_{r-1}}$. Hence, Lemma \ref{lem:fullsocle} applies with the faithful $\mathbb{F}_p[C_2]$-module $C_{p_r}$, yielding $\Gamma\supseteq C_{p_r}^{\deg(h)}$, as asserted in a).
\end{proof}

\subsection{An application to arithmetic dynamics}
\label{sec:dyn_appl}
Let $k$ be a field of characteristic $0$. Theorem \ref{thm:comp_dihedral} has an immediate consequence for {\it dynamical monodromy groups} $\varprojlim_n \Mon(f^n)$ of polynomials of prime degree. 
Recall that such $f$ are either $\AGL_1$-polynomials or have almost simple monodromy group, with a ``large kernel" conclusion for the latter case already having been obtained in \cite[Corollary 4.5]{KNR24}. Theorem \ref{cor:dynnom}  extends this by taking care of the case of polynomials $f$ linearly equivalent (over $\overline{k}$) to $T_p$.
\begin{proof}[Proof of Theorem \ref{cor:dynnom}]
By assumption, $\Mon_{\overline{k}}(f) = D_{p}$. Since $f$ is not conjugate to $\pm T_p$, its iterate $f\circ f$ is not linearly equivalent to $T_{p^2}$. The first assertion now follows from Theorem \ref{thm:comp_dihedral}a). Regarding the second one, Theorem \ref{thm:comp_dihedral} gives a lower bound of $p^{\sum_{k=0}^{n-1} p^k} \cdot 2^{p^{n-1}}$ for the order of $\Mon_{\overline{k}}(f^n)$, whereas $|[D_p]^n|=(2p)^{\sum_{k=0}^{n-1} p^k}$. 
Thus 
$$\frac{\log|G|}{\log|H|} \ge \frac{(\sum_{k=0}^{n-1} p^k) + \log_p(2) p^{n-1}}{(\sum_{k=0}^{n-1} p^k)\cdot (1+\log_p(2))},$$ and therefore 
\begin{align*}\liminf_{n\to\infty} \frac{\log|G|}{\log|H|} & \ge  \frac{1}{1+\log_p(2)} \liminf_{n\to\infty} (1 + \log_p(2)\frac{p^{n-1}(p-1)}{p^n - 1}). 
\end{align*}
Since the limit on the right side is $1+\log_p(2)(1-1/p)$, we get: 
$$
\liminf_{n\to\infty} \frac{\log|G|}{\log|H|}\ge \frac{p(1+\log_p(2)) - \log_p(2)}{p(1+\log_p(2))} = 1-\frac{\log_p(2)}{p(1+\log_p(2))}.
$$
\end{proof}
Note that Theorem \ref{cor:dynnom} applies in particular to many PCF (post-critically finite) polynomials with dihedral monodromy group. 
The asserted containment of $[C_p]^\infty$ in $\varprojlim_n \Mon(f^n)$  is then notably different from the case of PCF {\it unicritical} polynomials $f$. For these, one has  $[[C_p]^\infty: \varprojlim_n \Mon_{\overline{k}}(f^n)]=\infty$, see e.g., \cite{AH25}.
For the case $p=3$, Theorem \ref{cor:dynnom} in fact applies to all non-unicritical polynomials not conjugate to $\pm T_3$ (see \cite{CubicPCF} for a list of PCF cubics over $\mathbb{Q}$), and gives a value of approximately $0.871$ for the Hausdorff dimension.

\appendix

\section{Exceptional monodromy groups of $2$-step decomposable polynomials with small kernel}\label{app:A}
Here we collect the monodromy groups of polynomials $f=g\circ h$ with $g,h$ indecomposable,
such that $\ker(\Mon(f)\to \Mon(g))$ is not large in our sense, and moreover $g\circ h$ does not fall into Cases \ref{mainthm_xntn} or \ref{mainthm_ritt} of Theorem \ref{thm:mainres}; see in particular Theorems \ref{thm:affS4I} and \ref{thm:nonsaffI}. We give the geometric monodromy group $G:=\Mon_{\overline{k}}(f)$ and the arithmetic monodromy group $A:=\Mon_k(f)$.
Consider the normalizer of $G$ in $S_d$, $d=\deg f$, a.k.a.\ {\it its symmetric normalizer}. This contains $A$. Since all but three of the groups $G$ are their own symmetric normalizer, one automatically gets $G=A$; the other three have index $2$ in their symmetric normalizer, giving one additional option for $A$.

\begin{table}[h!]
\begin{tabular}{c|c|c|c|c}
	$\deg(g)$ & $\Mon_{\overline{k}}(h)$ & $G:=\Mon_{\overline{k}}(f)$ & comments & $A:=\Mon_k(f)$\\
	\hline
	$4$ & $C_2$ & $2.S_4\cong \GL_2(3)$ & nonsplit, $8T23$ & $G$\\
	\hline
	$6$ & $C_2$ & $2.\PGL_2(5)$ & nonsplit, $12T124$& $G$\\
	%$7$ & $C_2$ & $C_2\times PSL_3(2)$ & Ritt move\\
	$7$ & $C_2$ & $C_2\times \PSL_3(2)$ & $14T17$& $G$\\
	$7$ & $C_2$ & $2^4.\PSL_3(2)$ & nonsplit, $14T42$& $G$\\ 
	$7$ & $C_2$ & $2^4\rtimes \PSL_3(2)$ & $14T43$& $G$\\ 
	$8$ & $C_2$ & $2.\PGL_2(7)$ & nonsplit, $16T1036$& $G$ \\ 
	%$9$ & $C_2$ & $C_2\times P\Gamma L_2(8)$ & Ritt move\\
	$10$ & $C_2$& $2.\PGaL_2(9)$ & nonsplit, $20T265$& $G$\\
	$11$ & $C_2$ & $C_2\times M_{11}$ & $22T26$& $G$\\
	% $13$ & $C_2$  & $C_2\times PSL_3(3)$ & Ritt move\\ 
	$13$ & $C_2$  & $C_2\times \PSL_3(3)$ & $26T47$& $G$\\
	% $15$ & $C_2$ & $C_2\times PSL_4(2)$ & Ritt move\\
	$15$ & $C_2$ & $2^5\rtimes \PSL_4(2)$ & $30T1893$& $G$\\
	$15$ & $C_2$ & $2^{11}\rtimes \PSL_4(2)$ & $30T3819$& $G$\\
	$21$ & $C_2$ & $2^{10}\rtimes \PGaL_3(4)$ & $42T3846$& $G$\\
	$23$ & $C_2$ & $2^{12}\rtimes M_{23}$ & $46T39$& $G$\\
	$31$ & $C_2$ & $2^{16}\rtimes \PSL_5(2)$ &  & $G$\\ 
	$31$ & $C_2$ & $2^{26} \rtimes \PSL_5(2)$ & & $G$\\
	$5$ & $C_3$ & $C_3\times A_5$ & $15T15$ & $G$ or $15T21$\\
	%$7$ & $C_3$ & $C_3\times PSL_3(2)$ & Ritt move\\
	$8$ & $C_3$ & $C_3\times \PGL_2(7)$ & $24T2668$& $G$\\
	$11$ & $C_3$ & $3^6\rtimes \PSL_2(11)$ & $33T63$ & $G$ or $33T69$ \\
	%$13$ & $C_3$ & $C_3\times PSL_3(3)$ & Ritt move\\
	$13$ & $C_3$ & $3^7\rtimes \PSL_3(3)$ & $39T210$& $G$ or $39T218$\\
	$13$ & $D_3$ & $(3^{10}\rtimes 2)\rtimes \PSL_3(3)$ & $39T248$& $G$
\end{tabular}
\caption{Groups occurring as monodromy groups of length-$2$ decomposable polynomials $f=g\circ h$ with ``small kernel", not linearly equivalent to $X^n$ or $T_n$, and without a Ritt move.}
\label{table:sporcases}
\end{table}

In all but the first line of the table, $\Mon(f)$ is nonsolvable; moreover, in all cases one has $h\sim X^2$, $h \sim X^3$ or $h\sim T_3$. 
We furthermore keep track of whether the group extension $\Mon(f)\to\Mon(g)$ is split or not, and whenever possible give the precise label of $\Mon(f)$ in Magma's transitive group database (since, e.g., this information is needed in case of isomorphism $\Mon(f)\cong \Mon(g)\times \Mon(h)$, in order to exclude the presence of a Ritt move in  Proposition \ref{prop:rittimpldiag}). Note that the groups $\Mon(f)$ of the largest occurring degree $62$ are currently not in the scope of the transitive group database, although they may be uniquely identified by the fact that there is only a single conjugacy class of transitive subgroups of $C_2\wr \PSL_5(2)$ of each indicated type. We also note that, due to the explicitly known polynomials with exceptional nonsolvable monodromy groups (e.g., \cite{Mul2}, \cite{CC}, \cite{Elkies}) and due to the fact that only very specific ramification types can yield the monodromy groups in Table \ref{table:sporcases}, it is in principle possible to compute all polynomials giving rise to these monodromy groups, although the calculations would be very tedious. We mention, as a sample result, that the case $\Mon(f) = 15T15\cong C_3\times A_5$ occurs exactly for $f=g\circ h$ linearly equivalent to $(X^3(X^2+5X+40))\circ (X^3+\alpha)$ with $\alpha$ a root of $X^2+5X+40$. Indeed, the only way to produce this monodromy group is to compose the $A_5$-polynomial with ramification type $([5], [3,1^2], [2^2,1])$ with $X^3$ in such a way that the branch point $0$ of $X^3$ is one of the special points over the $[3,1^2]$ branch point. Up to linear equivalence, this yields the above polynomial. For some further explicit computational results, note that the cases with $\Mon_{\overline{k}}(f) = 8T23, 12T124, 15T15, 16T1036, 20T265, 24T2668$ all feature in the classification of pairs of Kronecker conjugate polynomials of composition length $2$ achieved in \cite{Mul}, with the explicit polynomials being computed in Section 4 of that paper.

Coincidentally, none of the polynomials $f$  corresponding to the groups in Table \ref{table:sporcases} can be defined over $\mathbb{Q}$ or even $\mathbb{R}$, which can most easily be seen from the fact that none of the arithmetic monodromy groups $A$ contain an element $x$ of the normalizer of a cyclic transitive subgroup $\langle\tau\rangle$ with $x\tau x^{-1}=\tau^{-1}$, something that would however be necessary for definability over $\mathbb{R}$ due to the action of complex conjugation on the inertia group at infinity, cf., e.g., \cite[Theorem I.10.3]{MM}.

\section{Some extensions of Lemma \ref{lem:fullsocle}}
For the sake of future applications to lower-bounding monodromy groups of compositions of arbitrarily many polynomials, we record here some generalizations of Lemma \ref{lem:fullsocle}, namely to modules $W'$ that need not be semisimple.

\begin{lem}
\label{lem:fullsocle_heart}
Let $H=W.U$, 
where $W$ is a faithful submodule of the $\mathbb{F}_p[U]$-permutation module.  
Let $G\le H\wr S_d$
have transitive image under $\pi:G\to S_d$. Let $\Gamma=\ker(\pi) = G\cap H^d$ and $\Delta  = \Gamma\cap W^d$. Furthermore, let $\rho:\Gamma\to \GL(W)^d$ be the  action of $\Gamma\leq H^d$ on $W^d$. Assume all of the following:
\begin{itemize}
	\item[i)] 
	The image of $\Gamma$
	under projection to (any) $H$-component equals $W.U'$ for $U'\le U$ such that  $W=W_1\oplus\dots \oplus W_r$ is a direct sum of indecomposable $\mathbb{F}_p[U']$-modules $W_j$, $j=1,\dots, r$. 
	Let $\rho_j:\Gamma\to \GL(W_j)^d$ be the  action of $\Gamma\leq H^d$ on $W_j^d$. 
	\item[ii)] Moreover $\tilde{W}_j := W_j/(W_j\cap D)$ is an irreducible $\mathbb{F}_p[U']$-module for all $j\in \{1,\dots, r\}$, where $D\subseteq W$ is the diagonal. 
	Let $\tilde\rho_j:\Gamma\to \GL(\tilde W_j)^d$ be the  action of $\Gamma\leq H^d$ on $\tilde W_j^d$. 
	
	\item[iii)] There exist $s\ge 1$ elements $x_1,\dots,x_s\in \Gamma$ 
	admitting intersections $\bigcap_{i=1}^s\supp(\rho(x_i))=\bigcap_{i=1}^s\supp(\tilde{\rho}(x_i))$ of size $1$.  Let $J$ be the set of all $j$ for which the intersection $\bigcap_{i=1}^s\supp(\tilde \rho_j(x_i))$ is nonempty (and hence of size $1$). 
	%\begin{equation}
	%\label{eq:support1_strong}
	%\bigcap_{m=1}^s\supp(\rho_j(x_m)) = %\bigcap_{m=1}^s\supp(\tilde{\rho}_j(x_m))\subseteq \{i\},
	%\end{equation} where $\rho_j(x)$ and $\tilde{\rho}_j(x)$ denote the image of $\rho(x)$ in $GL(W_j)^d$ and in $GL(\tilde{W}_j)^d$, respectively.
	%Need to add condition on all of W??
\end{itemize} 
Then $\Delta\supseteq W'^d$, where $W':= \bigoplus_{j\in J} W_j$.
\end{lem}
\begin{proof} 
Note that the irreducibility of $W_j$ was only used to guarantee that Observation 1 in the proof of Lemma \ref{lem:fullsocle} holds. Using instead the irreducibility assumption on $\tilde{W}_j$, this can be replaced by: 

\underline{Observation 1'}: Given any $\gamma \in \Gamma$ whose $i$-th component $\gamma_i$ lies in $W_j\setminus D$ and any $1\ne u\in U$ acting nontrivially on $\tilde{W}_j$, there exists a $\Gamma$-conjugate $\tilde{\gamma}\in \Gamma$ of $\gamma$ such that $\tilde{\gamma}_i$ lies in $W_j$ {\textit and} $\tilde{\gamma}_i^u \not\equiv \tilde{\gamma}_i \bmod D$.

The remainder of the proof can therefore be carried out in analogy with that of Lemma \ref{lem:fullsocle}. Since  $x_1,\dots, x_s$ all act nontrivially on $\tilde{W}_j$ for any $j\in J$,  Observation 1' ensures that  
all commutators formed in the proof have a component entry in $W_j\setminus D$. The proof is concluded by noting that, since the dimension of the diagonal submodule $W_j\cap D$ is at most $1$, the $\mathbb{F}_p[U']$-submodule of $W_j$ generated by any element of $W_j\setminus D$ is necessarily $W_j$ itself. 
Thus $\Delta\supseteq W_j^d$ follows as before.
\end{proof}

\begin{exa}
\label{ex:2s4}
In the setting of Lemma \ref{lem:fullsocle_heart}, let ${\rm Aug}(\mathbb{F}_2^4).S_4 \le H \le C_2\wr S_4$. Since the $S_4$-permutation module $\mathbb{F}_2^4$ has irreducible heart ${\rm Aug}(\mathbb{F}_2^4)/{\rm diag(\mathbb{F}_2^4)}$, Assumption ii) holds with $W'={\rm Aug}(\mathbb{F}_2^4)$. Moreover, any element of order $3$ in $S_4$ acts nontrivially on the heart. Thus, if $d\ge 3$ and $\Gamma/\Delta$ contains a subgroup $C_3^{d-1}$, Assumption iii) clearly holds as well (namely, with $s=2$ elements of order $3$, each supported on only two components), yielding the conclusion $\Delta\supseteq ({\rm Aug}(\mathbb{F}_2^4))^d$.
\end{exa}

A situation still not falling into the scope of Lemma \ref{lem:fullsocle_heart} is $H\le C_p\wr C_p$ ($p$ prime), since the $\mathbb{F}_p[C_p]$ permutation module is indecomposable with submodules of every dimension $0\le d\le p$. To obtain a reasonably simple criterion here, we restrict to a particular special case of Assumption iii).

\begin{lem}
\label{lem:fullsocle_indecomp}
Let $H=W.U$, where 
$W$ is an $\mathbb{F}_p[U]$-module.  
Let $G\le H\wr S_d$  
be such that $\pi:G\to S_d$ maps onto a transitive subgroup of $S_d$, and let $\Gamma=\ker(\pi) = G\cap H^d$. Furthermore, let $\rho:\Gamma\to U^d$ be  the natural projection and $\Delta = \ker(\rho) = \Gamma\cap W^d$. Assume all of the following:
\begin{itemize}
	\item[i)] The image of $\Gamma$
	under projection to one component contains $H':=W.U'$ for some subgroup $U'\le U$ such that $W$ is an indecomposable $\mathbb{F}_p[U']$-module.
	\item[ii)] $\Gamma/\Delta \supseteq U'^d$. 
\end{itemize} 
Set $W_0:=W$ and $W_i:=[W_{i-1},H']$ for $i\ge 1$, where $[W_{i-1},H']$ denotes the module generated by commutators.
Then the following hold:
\begin{itemize}
	\item[1)] $\Delta \supseteq W_2^d$, 
	\item[2)] If additionally $W_1 \subsetneq W$ is the unique maximal $\mathbb{F}_p[U']$-submodule of $W$, then either $W_1=W_2$ or
	$\Delta/W_2^d$ contains an element of support size $2$ in $(W_1/W_2)^d$. 
\end{itemize}
\end{lem}
\begin{proof}
Let $u,v\in U'$. Due to Assumption ii), there exist $x,y\in \Gamma$ with $\rho(x)=(u,1,\dots, 1)$ and $\rho(y)=(v,1,\dots, 1)$. In other words, $x=(\tilde{u}, w_2,\dots, w_d)$ for some $\tilde{u}\in H$ mapping to $u$ under the projection $H\to U$, and for some $w_2,\dots, w_d\in W$; and analogously for $y$. 
Moreover, for every $w\in W$, there exists $z\in \Gamma$ with first component $z_1=w$, due to Assumption i). Then the commutator $[z,x]$ is an element of $\Delta$ with first coordinate $[w,\tilde{u}]$. Thus, the commutator $[[z,x],y]\in \Delta$ is supported only on the first component, with entry $[[w,\tilde{u}],\tilde{v}]$. As $u,v$ run through all of $U'$ and $w$ through all $W$, these elements generate all of $W_2$, so that $\Delta \supseteq W_2^d$, showing 1). Next, consider the quotient module $\tilde{\Delta}:=\Delta/W_2^d$. 
We distinguish two cases, depending on the component entries of the element $x$ above.

{\it Case 1}: For all $x$ as above, it holds that $w_j\in W_1$ for all $j\in \{2,\dots, d\}$. Then all but the first component entry of $[z,x]$ are in fact in $W_2$. On the other hand, as soon as $W_1\subsetneq W$, we may pick $z$ such that its first component entry $w$ is in $W\setminus W_1$. The additional assumption of $W_1$ being the unique maximal submodule of $W$ implies that the $\mathbb{F}_p[U']$-submodule generated by $w$ is all of $W$. If the module $[w, H]$ generated by all commutators $[w, h]$, $h\in H'$ is contained in $W_2$, then $W_1=[W,H'] = [w, H'] = W_2$. We may thus assume that there exists $\tilde{u}\in H'$ with $[w, \tilde{u}] \notin W_2$,
so that for this choice of $w$ and $\tilde{u}$, the element $[z,x]$ maps to an element of $\tilde{\Delta}$ supported exactly on the first component, trivially implying 2).

{\it Case 2}: There exists $x$ as above and $j\in \{2,\dots, d\}$ with $w_j\notin W_1$. 
%The additional assumption of $W_1$ being the unique maximal submodule of $W$ implies that the $\mathbb{F}_p[U]$-submodule generated by $w_j$ is all of $W$. If the module $[w_j, H]$ generated by all commutators $[w_j, h]$, $h\in H$ is contained in $W_2$, then $W_1=[W,H] = [w_j, H] = W_2$. 
As in Case 1), we may  assume that there exists $h\in H'$ with $[w_j, h] \notin W_2$. By Assumption ii), we may pick $\tilde{x}\in \Gamma$ with $\supp(\rho(\tilde{x})) = \{j\}$ and whose $j$-th component $\tilde{x}_j$ has the same image as $h$ under projection $H\to U$.
Then $[\tilde{x},x]$ is supported at most on the first and the $j$-th component, and additionally the $j$-th component entry equals $[w_j, h] \notin W_2$. This again implies 2).
\end{proof}

\begin{rmk}
\label{rmk:stronger}
The commutator approach used in the above proof shows  that in fact $W_2^d \subseteq [\Gamma, \Gamma]\cap \Delta$, and even more precisely that $W_2^d\subseteq [\Gamma_0,\Gamma_0]\cap \Delta$ with $\Gamma_0\subseteq \Gamma$ the preimage of $U'^d\subseteq \Gamma/\Delta$.
\end{rmk}

\begin{exa}
\label{exa:cpcp}
For $W=\mathbb{F}_p^p$ the permutation module under the cyclic group $U'=C_p$ (i.e., $H'=C_p\wr C_p$), $W_j$ is the unique submodule of codimension $j$ in $W$; in particular, one has $W_1=\textrm{Aug}(\mathbb{F}_p^p)$ and $W_1/W_2\cong \mathbb{F}_p$.
In the case where $\pi(G)\le S_d$ additionally acts primitively (e.g., when $d$ is a prime), an element of support size $2$ necessarily generates a submodule of codimension $1$ in $(W_1/W_2)^d$. Lemma \ref{lem:fullsocle_indecomp} thus asserts that in this case $\Delta$ contains a submodule of codimension $\le 1$ of $(\textrm{Aug}(\mathbb{F}_p^p))^d$, and due to Remark \ref{rmk:stronger}, this submodule is even contained in $\Delta\cap [\Gamma_p,\Gamma_p]$ with a $p$-Sylow subgroup $\Gamma_p$ of $\Gamma$. 
\end{exa}

\bibliography{biblio2}
\bibliographystyle{amsalpha}
\end{document}